\documentclass[12pt]{amsart}     

\usepackage{amsmath}
\usepackage{amsthm}
\usepackage{amsfonts}
\usepackage{latexsym}
\usepackage{float}
\usepackage{upref}
\restylefloat{figure}
\usepackage{mathptmx}
\usepackage[dvips]{graphicx}
\usepackage{comment}

\usepackage{mathptmx}
\usepackage{times}

\newtheorem{theorem}{Theorem}[section]
\newtheorem{lem}[theorem]{Lemma}
\newtheorem{cor}[theorem]{Corollary}
\newtheorem{prop}[theorem]{Proposition}
\theoremstyle{definition}
\newtheorem{defn}[theorem]{Definition}
\theoremstyle{remark}
\newtheorem{rem}[theorem]{Remark}

\newtheorem{theorem*}{Theorem}
\newtheorem{assumption*}[theorem*]{Assumption}

\usepackage{latexsym} 
\usepackage{amsmath}
\usepackage{amsthm}
\usepackage{amsfonts}
\usepackage{amssymb}
\usepackage{amsxtra}
\usepackage{amscd}
\usepackage{ifthen}
\usepackage{graphicx}
\usepackage{enumerate}
\usepackage[hidelinks, colorlinks, allcolors = black, breaklinks = true, pagebackref=false]{hyperref}
\renewcommand{\ge}{\geqslant}

\renewcommand{\le}{\leqslant}

\usepackage[usenames,dvipsnames,svgnames,table]{xcolor}

\usepackage[colorinlistoftodos,prependcaption,color=yellow,textsize=tiny,textwidth=2cm]{todonotes}
\usepackage[utf8]{inputenc}    
\usepackage[T1]{fontenc}       

\usepackage[margin=1in]{geometry}

\hypersetup{
pdfauthor={Benjamin Gess, Jonas Sauer},
pdftitle={Optimal regularity in time and space for nonlocal porous medium type equations},
breaklinks=true,
colorlinks=true,
linkcolor=blue,
citecolor=blue,
urlcolor=blue,
filecolor=blue,
}

\numberwithin{equation}{section} 

\newcommand{\eps}{\varepsilon}
\newcommand{\dd}{\,\mathrm{d}}
\newcommand{\ddd}{\,\textrm{\dj}}

\newcommand{\q}{{{q}}}
\newcommand{\qq}{Q}

\newcommand{\td}{\tilde}
\newcommand{\vp}{\varphi}
\newcommand{\ie}{\emph{i.e.}}
\renewcommand{\div}{\mathrm{div}\,}     

\renewcommand{\Re}{\mathop{\mathrm{Re}}}
\renewcommand{\Im}{\mathop{\mathrm{Im}}}

\newcommand{\he}{\upsilon}
\newcommand{\nonl}{\Phi}
\newcommand{\me}{\textsf{m}}
\newcommand{\1}{\mathsf{1}}
\newcommand{\uf}{u}
\newcommand{\ubb}{u_{BB}}

\newcommand{\vf}{v}

\newcommand{\f}{f}
\newcommand{\g}{g}

\renewcommand{\t}{\theta}

\newcommand{\tauz}{z}

\newenvironment{pdeq}{ \left\{ \begin{aligned}}{\end{aligned}\right.}

\newcommand{\vpp}[1]{\ensuremath{\left(#1\right)}}
\newcommand{\calf}{{\mathcal F}}

\newcommand{\calh}{{\mathcal H}}

\newcommand{\call}{{\mathcal L}}

\newcommand{\calm}{{\mathcal M}}

\newcommand{\calo}{{\mathcal O}}

\newcommand{\calr}{{\mathcal R}}
\newcommand{\cals}{{\mathcal S}}

\newcommand{\R}{\mathbb{R}}
\newcommand{\Z}{\mathbb{Z}}
\newcommand{\C}{\mathbb{C}}

\newcommand{\N}{\mathbb{N}}

\DeclareMathOperator{\id}{\textsf{id}}

\DeclareMathOperator{\supp}{supp}
\DeclareMathOperator*{\esssup}{ess\,sup}

\DeclareMathOperator{\sgn}{sgn}

\newcommand{\diffop}{\call}

\newcommand{\set}[1]{\ensuremath{\{#1\}}}

\newcommand{\setc}[2]{\ensuremath{\{#1\ \lvert\ #2\}}}

\newcommand{\frex}{a}

\renewcommand{\d}{\delta}

\newcommand{\op}{\mathsf{op}\,}
\newcommand{\norm}[1]{\lVert#1\rVert}

\newcommand{\opnorm}[1]{{\lvert\kern-0.25ex\lvert\kern-0.25ex\lvert #1 \rvert\kern-0.25ex\rvert\kern-0.25ex\rvert}}

\newcommand{\WSR}[2]{W^{#1,#2}}

\newcommand{\DSR}[2]{\dot{W}^{#1,#2}} 
 
\newcommand{\WSRloc}[2]{W^{#1,#2}_{\mathrm{loc}}} 
\newcommand{\CR}[1]{C^{#1}}
 
\newcommand{\LR}[1]{L^{#1}}

\newcommand{\LRloc}[1]{L^{#1}_{\mathrm{loc}}} 
\newcommand{\CRi}{\CR \infty}
\newcommand{\CRci}{\CR \infty_c}

\newcommand{\tin}{\text{in }}

\newcommand{\ton}{\text{on }}

\newcommand{\tand}{\text{and }}

\newcommand{\newCCtr}[2][d]{
\newcounter{#2}\setcounter{#2}{0}
\expandafter\xdef\csname kyedtheconst#2\endcsname{#1}
}
\newcommand{\Cc}[2][nolabel]{
\stepcounter{#2}
\expandafter\ensuremath{\csname kyedtheconst#2\endcsname_{\arabic{#2}}}
\ifthenelse{\equal{#1}{nolabel}}
{}
{\expandafter\xdef\csname kyedconst#1\endcsname
{\expandafter\ensuremath{\csname kyedtheconst#2\endcsname_{\arabic{#2}}}}}
}
\newcommand{\Ccn}[2][nolabel]{
\expandafter\ensuremath{\csname kyedtheconst#2\endcsname}
\ifthenelse{\equal{#1}{nolabel}}
{}
{\expandafter\xdef\csname kyedconst#1\endcsname
{\expandafter\ensuremath{\csname kyedtheconst#2\endcsname}}}
}

\newcommand{\Cclast}[1]{
\expandafter\ensuremath{\csname kyedtheconst#1\endcsname_{\arabic{#1}}}
}

\newcommand{\Ccllast}[1]{
\addtocounter{#1}{-1}
\expandafter\ensuremath{\csname kyedtheconst#1\endcsname_{\arabic{#1}}}
\addtocounter{#1}{1}
}
\newcommand{\const}[1]{
\expandafter{\ifcsname kyedconst#1\endcsname
  \csname kyedconst#1\endcsname
\else
  \errmessage{Undefined Kyedconstant #1.}%
\fi}
}

\makeatletter
\def\blfootnote{\xdef\@thefnmark{}\@footnotetext}
\makeatother

\theoremstyle{plain}
\newtheorem{thm}{Theorem}[section]
\newtheorem*{thm*}{Theorem}
\newtheorem{assumption}[thm]{Assumption}
\theoremstyle{remark}
\newtheorem{example}[thm]{Example}

\begin{document}                        


\title[Optimal regularity for nonlocal porous medium equations]{Optimal regularity in time and space\\for nonlocal porous medium type equations}


\author{Benjamin~Gess}
\address{Benjamin~Gess\newline
Max Planck Institute for Mathematics in the Sciences\newline Inselstr.~22, 04103 Leipzig, Germany\newline
and Fakult\"at f\"ur Mathematik, Universit\"at Bielefeld\newline Universit\"atstr. 25, 33615 Bielefeld, Germany }
\email{bgess@math.uni-bielefeld.de}
\author{Jonas~Sauer}
\address{Jonas~Sauer\newline
Friedrich-Schiller-Universit\"at Jena\newline Ernst-Abbe-Platz 2, 07737 Jena, Germany}
\email{jonas.sauer@uni-jena.de}


\begin{abstract} A broad class of possibly non-unique generalized kinetic solutions to hyperbolic-parabolic PDEs is introduced. Optimal regularity estimates in time and space for such solutions to nonlocal, and spatially inhomogeneous variants of the porous medium equation are shown in the scale of Sobolev spaces.
The optimality of these results is shown by comparison to the non-local Barenblatt solution.
The regularity results are used in order to obtain existence of generalized kinetic solutions.
\end{abstract}


\date{\today}

\maketitle   


\setcounter{tocdepth}{1}
\tableofcontents

\section{Introduction}\label{Int}
We prove optimal estimates on the time and space regularity in Sobolev spaces of solutions to equations of the form
\begin{align}\label{gen_sys}
 \begin{pdeq}
  \partial_t\uf+L\nonl(\uf) &= S && \tin (0,T)\times \R^d,\\
  \uf(0)&=\uf_0 && \tin \R^d,
 \end{pdeq}
\end{align}
where $u_0 \in L^1(\R^d)$, $S\in L^1((0,T)\times \R^d)$, $\nonl: \R\to\R$, $a\in (0,2)$, and $L$ is a linear integro-differential operator given by
\begin{align}\label{def_op}
L\uf(x):=-\int_{\R^n} \big(\uf(x+y)-\uf(x) - \1_{[1,2)}(\frex)\1_{|y|\le 2}(y) y\cdot \nabla\uf(x)\big)k(x,y) \dd y.
\end{align}

In the recent contributions \cite{Ges21,GST20}, the optimal regularity in space and time for solutions to local, homogeneous porous media equations has been obtained.
Some of the key ingredients of these works were kinetic solution theory, singular moments of entropy dissipation measures, and velocity averaging in terms of micro-local Fourier analysis.
In the present work, we extend this theory in several directions:

(1) We treat non-local, degenerate PDEs, such as the non-local porous medium equation.

(2) We include spatially inhomogeneous degenerate diffusions.

(3) The previous works \cite{Ges21,GST20} relied on the uniqueness of solutions in order to derive singular moment estimates on the entropy dissipation measures. In this work, we introduce a new modification of the approach which completely avoids this.
As a consequence, the regularity statements apply to a large class of -- possibly non-unique -- solutions. 

Notably, the regularity estimates obtained in the present work depend on $u_0$ only via its $L^1$ norm. In particular, they immediately generalize to measure-valued initial data.

We will always work under the following assumption.
\begin{assumption}\label{ass_1}
\begin{enumerate}
    \item\label{ass_1i} The kernel $k(x,y):\R^d\times (\R^{d}\setminus\set{0})\to (0,\infty)$ is
    smooth in $x$ and $y$, and such that for all $\alpha,\beta\in\N_0^d$ there are $C_{\alpha,\beta}>0$ and $c>0$ with
\begin{align}
|\partial_y^\beta \partial_x^\alpha k(x,y)|&\le C|y|^{-d-\frex-|\beta|}, && \forall |y|>0, \label{ass_ker_1} \\
k(x,y)&\ge c|y|^{-d-\frex}, && \forall |y|>0, \ x\in \R^d \label{ass_ker_2}.
\end{align}
Moreover, we assume $k(x,y)=k(x+y,-y)$, and in the case $\frex=1$ also $k(x,y)=k(x,-y)$.
\item\label{ass_1ii} For the nonlinearity we suppose that $\nonl\in \CR{1}(\R)\cap \WSRloc{2}{1}(\R)$ with $\nonl'\ge 0$,
and there are $C>0$ and $m\in(1,\infty)$ such that for all $\delta>0$
 \begin{align}
  |\set{\nonl'(\vf)\le \delta}|&\le C\delta^{\frac{1}{m-1}} \quad \text{and} \quad \#\set{\nonl'(\vf)=\delta}\le C. \label{ass_nonl_1}
 \end{align}
 where we have written for short $\set{\nonl'(\vf)\le\d}:=\setc{\vf\in \R}{\nonl'(\vf)\le\d}$.
\end{enumerate}
\end{assumption}

The condition $k(x,y)=k(x+y,-y)$ ensures that $L$ is $L^2$-symmetric, and that $\int_{\R^d} Lu(x) \dd x=0$.

\medskip

A prime example which satisfies these conditions is the case of a fractional porous medium equation, \ie~$\nonl(v):=v^{[m]}:=|v|^{m-1}v$, $m>1$ and $k(x,y)=c|y|^{-d-\frex}$, which corresponds to the fractional Laplacian $L=(-\Delta)^{\frac{\frex}{2}}$. The corresponding problem then takes the form
 \begin{align}\label{pme_sys}
 \begin{pdeq}
  \partial_t\uf+(-\Delta)^{\frac{\frex}{2}}\uf^{[m]} &= S && \tin (0,T)\times \R^d,\\
  \uf(0)&=\uf_0 && \tin \R^d.
 \end{pdeq}
 \end{align} 
 We discuss some further examples of admissible nonlinearities $\nonl$ in Example \ref{ex_nonl} below.
 
\medskip

  The arguments of the present work rely on velocity averaging techniques and the kinetic formulation
\begin{align}
\partial_{t}\chi+\nonl'L_x\chi & =\partial_{\vf}q+S\delta_{v=u(t,x)},\label{pme_sys_kin}
\end{align}
for a non-negative measure $q$, where $\chi(u(t,x),v):=\1_{\vf<\uf(t,x)}-\1_{\vf<0}$. However, in contrast to \cite{GST20} we can allow even for generalized kinetic solutions $f$ which are not confined to the specific form of $\chi$, cf.~Definition \ref{def:gen_kinetic_sol} below. Additionally, both of the key arguments, that is, the micro-localization and the treatment of the entropy dissipation measure proceed along different lines, which allows to significantly weaken the assumptions. In particular, we do not rely on singular moments of the measure $q$ at all. Indeed, it will be sufficient to exploit that $q$ is bounded in the velocity variable $v$, which we prove to be satisfied for all generalized kinetic solutions.

The non-locality and spatial inhomogeneity of the operator $L$, and the generality of the nonlinearity $\nonl$ in \eqref{gen_sys} cause several challenges requiring new arguments, which we will comment on next:

{(1)} Due to the spatial inhomogeneity of \eqref{gen_sys}, the established methods from \cite{Ges21,GST20} do not appear to be applicable, and a different line of arguments needs to be developed.
Indeed, \cite{Ges21,GST20} relied on the so-called truncation property for the kinetic operator $\nonl'(v)L_x$ in \eqref{pme_sys_kin}, which corresponds to a microlocal analysis relating the Fourier variable $\xi$ with the velocity variable $v$.
However, in the setting of the present work, due to the spatial inhomogenity and non-locality of the operator $L$ it is not possible anymore to directly link the velocity $v$ to the Fourier variable $\xi$. Therefore, we develop a different approach to the microlocal analysis of the kinetic operator in Section \ref{AvLem}, which does not rely on a truncation property anymore.

{(2)} The $v$-dependency of the kinetic operator $\nonl'(v)L_x$ requires an extension and adaptation of the parametrix method for pseudodifferential operators providing precise bounds in terms of $v$ (see Lemma \ref{lem:symbol_est} and \ref{lem:FM}). 

{(3}) 
Since the parametrix does not provide an exact inverse, lower order terms emerge on the right-hand side of the estimates, which need to be controlled.
This is a delicate issue: Since we work in the broad class of (possibly non-unique) generalized kinetic solutions, only the properties contained in their definition can be used. This is different to previous work where the uniqueness of solutions was known, and, therefore, additional facts about solutions could be used. Moreover, the lower order terms contain the non-linearity $\nonl$, and, thus, are superlinear.
Therefore a careful interpolation and absorption argument is necessary to guarantee that these terms can be controlled for a large class of non-linearities (see Step \ref{cor:pme_l1_prf_st5} in the proof of Theorem \ref{cor:pme_l1}).

{(4)} 
Generalizing the decomposition employed in \cite{GST20} in terms of small and large velocities, we decompose in terms of level sets of $\nonl'$, and we prove that this allows the treatment of more general nonlinearities than in \cite{Ges21,GST20}.

{(5)} 
Finally, we identify a specific class of inhomogeneous Besov spaces of dominating mixed smoothness (see Lemma \ref{lem:interpol}) to be particularly well suited for the theory developed in this work. This allows to avoid the use of homogeneous spaces, which in \cite{GST20} required a careful argument in order to return to standard inhomogeneous (Sobolev) spaces. The approach presented here is more direct and makes use of recent advances in interpolation theory based on wavelet descriptions \cite{Tri19}.
In the argument for the regularity in time, cf. Lemma \ref{lem:av3}, interpolation theory is avoided altogether.

\medskip

Our first main result is the following.
 
 \begin{thm}\label{cor:pme_l1}
 Let $u_0\in\LR{1}(\R^d)$, $S\in L^1((0,T)\times \R^d)$, and let $L$ and $\nonl$ satisfy Assumption \ref{ass_1} for some $\frex\in(0,2)$, $m\in(1,\infty)$. Let $f$ be a generalized kinetic solution to \eqref{gen_sys} on $[0,T]\times \R^d$, cf.~Definition \ref{def:gen_kinetic_sol}, $\eta \in \CRi(\R_v)$ such that $\eta \nonl' f\in \LR{1}_{t,x,v}$, $|\eta|\le 1$ and $\eta'\in \LR{1}(\R_v)$, and write $u^\eta(t,x):=\int_{\R_v} \eta(v) f(t,x,v)\dd v$, $u(t,x):=u^\mathbf{1}(t,x)$.
  Let $p\in (1,m)$ and define
 \begin{align*}
  \kappa_t:= \frac{m-p}{p}\frac{1}{m-1}, \quad
  \kappa_x:= \frac{p-1}{p}\frac{\frex}{m-1}.
 \end{align*}
  Then, for all $\sigma_t\in [0,\kappa_t)$ and $\sigma_x\in [0,\kappa_x)$, $\eta \in \CRci(\R_v)$  we have
 \begin{align*}
  u^\eta\in \WSR{\sigma_t}{p}(0,T;\WSR{\sigma_x}{p}(\R^d)).
 \end{align*}
  Moreover, we have the estimate
 \begin{align}\label{cor:pme_est1_l1}
  \|u^\eta\|_{\WSR{\sigma_t}{p}(0,T;\WSR{\sigma_x}{p}(\R^d))}^p\lesssim (1+\|\eta'\|_{\LR{1}_v})(\|u_0\|_{L^1_{x}}^m +\|S\|_{L^1_{t,x}}^m + 1) +\|\eta\nonl'  f\|_{\LR{1}_{t,x,v}}.
 \end{align}
 Here, the implicit constant depends only on $d$, $\frex$, $m$, $p$, and the constants $c$, $C$ and $C_{\alpha,\beta}$ from Assumption \ref{ass_1} with $|\alpha|,|\beta|\le N$ for some $N=N(d)\in\N$.

 \medskip
 
  If $f(t,x,v)=\chi(u(t,x),v)$ for a.e. $(t,x,v)$ and $|\nonl|(u)\lesssim |u|+|u|^m$ then for any $\eps>0$ and all $p\in (1,m]$ we have
  \begin{align}\label{cor:pme_est1_l1_2}
  \|u\|_{\WSR{\sigma_t}{p}(0,T;\WSR{\sigma_x}{p}(\R^d))}^p\lesssim \|u_0\|_{L^1_{x}}^{m+\eps} +\|S\|_{L^1_{t,x}}^{m+\eps} +1,
 \end{align}
 where the implicit constant depends additionally on $\eps$ and where $\sigma_t=0$ and $\sigma_x\in (0,\frac{\frex}{m})$ in case of $p=m$.
 \end{thm}

By \cite{AAO20}, there is a unique kinetic solution to \eqref{pme_sys}. Theorem \ref{cor:pme_l1} hence immediately implies
 
 \begin{cor}\label{cor:pme_l1_intro}
 Let $u_0\in\LR{1}(\R^d)$, $S\in L^1((0,T)\times \R^d)$ and assume $\frex\in(0,2)$.
 Let $u$ be the unique kinetic solution to \eqref{pme_sys} with $m\in(1,\infty)$ on $[0,T]\times \R^d$.
  Let $p\in (1,m]$ and define $\kappa_t, \kappa_x$ as in Theorem \ref{cor:pme_l1} (and $\kappa_t=0$ and $\kappa_x=\frac{\frex}{m}$ in case of $p=m$).
  Then, for all $\sigma_t\in [0,\kappa_t) \cup \set{0}$ and $\sigma_x\in [0,\kappa_x)$ we have $u\in \WSR{\sigma_t}{p}(0,T;\WSR{\sigma_x}{p}(\R^d))$, and for all $\eps>0$ there holds
  \begin{align*}
  \|u\|_{\WSR{\sigma_t}{p}(0,T;\WSR{\sigma_x}{p}(\R^d))}^p\lesssim \|u_0\|_{L^1_{x}}^{m+\eps} +\|S\|_{L^1_{t,x}}^{m+\eps} +1,
 \end{align*}
 where the implicit constant depends only on $d$, $\frex$, $m$, $p$ and $\eps$.
 \end{cor}
 The optimality of these estimates is shown in Section \ref{Scale} below. The assertions in Theorem \ref{cor:pme_l1} and Corollary \ref{cor:pme_l1_intro} can be strengthened in the case of a bounded domain $\calo\subset\subset \R^d$.
 In the situation of Corollary \ref{cor:pme_l1_intro} we obtain for all $q\in[1,p]$ by the embedding $\LR{p}\subset \LR{q}$ on bounded sets, that $u\in \WSR{\sigma_t}{q}(0,T;\WSR{\sigma_x}{q}(\calo))$ with
 \begin{align}\label{cor:pme_est2_l1}
  \|u\|_{\WSR{\sigma_t}{q}(0,T;\WSR{\sigma_x}{q}(\calo))}\lesssim \|u_0\|_{L^1_x}^{m+\eps} +\|S\|_{L^1_{t,x}}^{m+\eps} +1.
 \end{align}
 Similar statements apply in case of Theorem \ref{cor:pme_l1}.

 \medskip

Due to the spatial inhomogeneity of \eqref{gen_sys}, the standard apriori estimates used in the construction of kinetic solutions, which are based on the $L^1$-contraction and spatial homogeneity, cannot be used. This demonstrates the relevance of the regularity estimates provided by Theorem \ref{cor:pme_l1} to the proof of existence of solutions, which constitutes the second main result of this work.
 \begin{thm}
\label{prop:wp-kinetic}Let $u_{0}\in L^{1}(\R^{d})$, $S\in L^{1}([0,T]\times\R^{d})$, and assume Assumption \ref{ass_1}. Then there exists a generalized kinetic solution $f$ to \eqref{gen_sys}. In addition, with $u(t,x):=\int f(t,x,v)dv$, $f$ satisfies $f(t,x,v)=\chi(u(t,x),v)$ for almost all $(t,x,v)\in [0,T]\times \R^d\times \R$.
\end{thm}
We conclude with explicit examples of nonlinearities $\nonl$ satisfying Assumption \ref{ass_1}\eqref{ass_1ii}.
 \begin{example}\label{ex_nonl}
 Let $\nonl\in C^1(\R)$ with $\nonl'\ge 0$. Each of the following conditions imply Assumption \ref{ass_1}\eqref{ass_1ii} with $m\in (1,\infty)$. 
 \begin{enumerate}
 \item\label{ex_nonli} $m$ is an integer and $\nonl$ is a polynomial of degree $m$. 
 \item\label{ex_nonlii} There are $c>0$, $N\in \N$ and $m_i\in (1,\infty)$, $v_i\in \R$, $i\in \set{1,\ldots,N}$, with
 \begin{align*}
 \nonl'(v)=c\prod_{i=1}^N |v-v_i|^{m_i-1} \quad \tand \quad m-1=\sum_{i=1}^N (m_i-1).
 \end{align*}
 \end{enumerate}
 \end{example}

\textbf{Discussion of the literature:}
In \cite{PQRV17}, for fractional porous medium equations (PME) the inner H\"older regularity of $\partial_t u$ and $(-\Delta)^{\frex/2}\nonl(u)$ is shown for $S=0$ under the constraint $\nonl\in C^{1,\gamma}(\R)$, $\gamma\in(0,1)$. The regularity of fractional PMEs on manifolds in the context of $L^p$-spaces has been considered in the works \cite{RoS22a,RoS22b}. In \cite[Theorem 6.2]{RoS22b} the authors showed under the condition $S\in C([0,T],L^p(M))$ and $u_0\in B^{2\frex-2\frex/p}_{p,q}(M)$ short time existence and uniqueness of a solution $u\in L^q(0,T;H^{2\frex}(M))\cap H^{1,q}(0,T;L^{p}(M))$, albeit under the additional assumption that the initial condition satisfies $u_0\ge c>0$, in which case the machinery of quasilinear evolution equations by Pr\"u{\ss} \cite{Pru02} and Cl\'ement-Li \cite{ClL93} based on maximal $L^p$-theory is applicable.

A theory of bounded entropy solutions to nonlocal diffusions has been developed in a series of works, see \cite{CJ2011,EJ2014,KU2011} and the references therein.
This has been extended to renormalized entropy solutions and $L^1$ data for nonlocal PME in \cite{ODI2021}.
A kinetic theory for nonlocal PME has been introduced in \cite{AAO20,JJG2018}.
In the present work, we extend  the concept of generalized kinetic solutions that was previously known in the context of local scalar conservation laws to the setting of nonlocal parabolic-hyperbolic PDEs.
In particular, this comprises a proof of their existence for $L^1$ forcings.

Several forms of nonlocal extensions of the porous medium equation (PME) have been discussed in the literature, see, for example, \cite{BKM2010,CV2011,dPQRV2011}.
For derivations of nonlocal PMEs from microscopic dynamcis we refer to \cite{CdPG2023,CHJZ2022}.
Many works have been devoted to the analysis of the asymptotic behavior and apriori estimates, see, for example \cite{BV2015,Vaz14} and the references therein.
Numerical methods for nonlocal PMEs have been analyzed, for example, in \cite{DT2014}.
For an overview of the literature up to 2017 we refer to \cite{CDPFMV2017}.

For literature on the (spatial) regularity of solutions to the local porous medium equations in Sobolev spaces we refer to \cite{Ebm05,Ges21,TaT07}, see also \cite{Ges21} for a more detailed account on the available literature in this regard.
Higher integrability has been shown in \cite{BDKS18,GiS19}.
Improved averaging estimates for scalar  conservation laws based on singular moments of the entropy dissipation measure have been first derived in \cite{GL2019}.
An approach to the compactness and regularity of solutions to local, inhomogeneous, parabolic-hyperbolic PDEs based on the concept of $H$-measures has been developed in \cite{EM2023,EM2023-2}.

For an overview of results concerning the Sobolev regularity of solutions to local, nonlinear $p$-Laplace type equations by means of  nonlinear Calder\'on-Zygmund theory we refer to \cite{Min10} and the references therein.
Hölder regularity for nonlocal, nonlinear $p$-Laplace type equations has been considered for example in \cite{BLS2018}.
Also the regularity of solutions to linear, inhomogeneous, nonlocal PDEs has attracted significant interest, see \cite{AbG23, C17, G2018, MSY2021} and the references therein.

\section{Preliminaries, Notation and Function Spaces}\label{FctSpc}

We use the notation $a\lesssim b$ if there is a universal constant $C>0$ such that $a\le Cb$. We introduce $a\gtrsim b$ in a similar manner, and write $a \sim b$ if $a\lesssim b$ and $a\gtrsim b$.
For a Banach space $X$ and $I\subset \R$ we denote by $C(I;X)$ the space of bounded and continuous $X$-valued functions endowed with the norm $\|f\|_{C(I;X)}:=\sup_{t\in I}\|f(t)\|_X$. If $X=\R$ we write $C(I)$. For $k\in \N\cup\set{\infty}$, the space of $k$-times continuously differentiable functions is denoted by $C^k(I;X)$. The subspace of $C^k(I;X)$ consisting of compactly supported functions is denoted by $C^k_c(I;X)$. Moreover, we write $\calm_{TV}$ for the space of all measures with finite total variation.
Throughout this article we use several types of $L^p$-based function spaces.
For a Banach space $X$ and $p\in[1,\infty]$, we endow the Bochner-Lebesgue space $\LR{p}(\R;X)$ with the usual norm
\begin{align*}
 \norm{f}_{\LR{p}(\R;X)}&:=\left(\int_{\R}\norm{f(t)}_{X}^p \dd t\right)^\frac{1}{p},
\end{align*}
with the standard modification in the case of $p=\infty$. For $k\in \N_0:=\N\cup\set{0}$, the corresponding $X$-valued Sobolev space is denoted by $\WSR{k}{p}(\R;X)$. If $\sigma\in (0,\infty)$ is non-integer (say $\sigma=k+r$, with $k\in\N_0$ and $r\in(0,1)$), then we define the $X$-valued Sobolev-Slobodecki\u{\i} space $\WSR{\sigma}{p}(\R;X)$ as the space of functions in $\WSR{k}{p}(\R;X)$ with
\begin{align}\label{hom_sob_norm}
 \norm{f}_{\DSR{\sigma}{p}(\R;X)}:=\left(\int_{\R\times\R}\frac{\norm{D^kf(t)-D^kf(s)}_X^p}{|t-s|^{rp+1}}\dd s \dd t\right)^{\frac{1}{p}}<\infty,
\end{align}
again with the usual modification in the case of $p=\infty$. Further, let $\DSR{\sigma}{p}(\R;X)$ be the space of all locally integrable $X$-valued functions $f$ for which \eqref{hom_sob_norm} is finite. If we factor out the equivalence relation $\sim$, where $f\sim g$ if $\norm{f-g}_{\DSR{\sigma}{p}(\R;X)}=0$, the space $\DSR{\sigma}{p}(\R;X)$ equipped with the norm $\norm{\cdot}_{\DSR{\sigma}{p}(\R;X)}$ is a Banach space.

Moreover, in order to treat regularity results in both time and space efficiently, we introduce spaces with dominating mixed derivatives set in the framework of Fourier analysis, that is, corresponding Besov spaces. These spaces have a long history in the literature, beginning with the works of S.\@ M.\@ Nikol'ski\u{\i} \cite{Nik62, Nik63b, Nik63a}. We refer the reader to the monograph of Schmeisser and Triebel \cite{ScT87} and the references therein. We adopt the notation of \cite{ScT87} for the non-homogeneous spaces.

Let $\vp$ be a smooth function supported in the annulus $\{\xi\in\R^{d}:\,\frac{1}{2}\le|\xi|\le2\}$ and such that
\[
\sum_{j\in\Z}\vp_j(\xi):=\sum_{j\in\Z}\vp(2^{-j}\xi)=1,\quad\forall\xi\in\R^{d}\setminus\set{0}.
\]
Similarly, let $\eta$ be a smooth function supported in $\vpp{-2,-\frac12}\cup \vpp{\frac12,2}$ with
\[
\sum_{l\in\Z}\eta_l(\tau):=\sum_{l\in\Z}\eta(2^{-l}\tau)=1,\quad\forall\tau\in\R\setminus\set{0}.
\]
Moreover, define $\phi_j:=\vp_j$ for $j\ge 1$ and $\phi_0:=1-\sum_{j\ge 1}\phi_j$ as well as $\psi_l:=\eta_l$ for $l\ge 1$ and $\psi_0:=1-\sum_{l\ge 1}\eta_l$. We will use the shorthand notation $\eta_l\vp_j$ for the function $(\tau,\xi)\mapsto \eta_l(\tau)\vp_j(\xi)$, and similarly for combinations of $\psi_l$ and $\phi_j$.

\begin{defn}\label{defn:spaces}
 Let $\sigma_i\in\vpp{-\infty,\infty}$, $i=t,x$, and $p,q\in[1,\infty]$. Set $\overline{\sigma}:=(\sigma_t,\sigma_x)$.
 The non-homogeneous Besov space with dominating mixed derivatives $S^{\overline{\sigma}}_{p,q}B(\R^{d+1})$ is given by
  \begin{align*}
   S^{\overline{\sigma}}_{p,q}B:=S^{\overline{\sigma}}_{p,q}B(\R^{d+1}):=\setc{f\in\cals'(\R^{d+1})}{\|f\|_{S^{\overline{\sigma}}_{p,q} B}<\infty},
  \end{align*}
  with the norm
  \begin{align*}
   \|f\|_{S^{\overline{\sigma}}_{p,q}B}&:=\left(\sum_{l,j\ge 0} 2^{\sigma_t lq}2^{\sigma_x jq}\|\calf^{-1}_{t,x}\psi_l \phi_j\calf_{t,x} f\|_{\LR{p}(\R^{d+1})}^q\right)^\frac1q, \quad q<\infty,\\
   \|f\|_{S^{\overline{\sigma}}_{p,\infty}B}&:=\sup_{l,j\ge 0} 2^{\sigma_t l}2^{\sigma_x j}\|\calf^{-1}_{t,x}\psi_l \phi_j\calf_{t,x} f\|_{\LR{p}(\R^{d+1})}.
  \end{align*}
\end{defn}

\begin{rem}\label{rem:spaces}
 The spaces $S^{\overline{\sigma}}_{p,q}B(\R^{d+1})$ are Banach spaces with $S^{\overline{\sigma}+\overline{\eps}}_{p,1}(\R^{d+1})\subset S^{\overline{\sigma}}_{p,\infty}B(\R^{d+1})$ for $\overline\eps\in (0,\infty)^2$, see Theorem 2.2.4 and Proposition 2.2.3/2(ii) in \cite{ScT87}.
\end{rem}
We will use the following interpolation result.
\begin{lem}\label{lem:interpol}
    Let $p_0,p_1,q_0,q_1\in [1,\infty]$ be such that $\min\set{q_0,q_1}<\infty$ and $\overline{r}_0,\overline{r}_1\in \R^2$.
    For $\theta\in (0,1)$ there holds the complex interpolation
  \begin{align*}
      [S^{\overline{r}_0}_{p_0,q_0}B,S^{\overline{r}_1}_{p_1,q_1}B]_\theta=S^{\overline r_\theta}_{p_\theta,q_\theta}B
  \end{align*}
  with $\overline{r}_\theta:=\overline{r}_0(1-\theta)+\overline{r}_1\theta$, $\frac{1}{p_\theta}:=\frac{1-\theta}{p_0}+\frac{\theta}{p_1}$, and $\frac{1}{q_\theta}:=\frac{1-\theta}{q_0}+\frac{\theta}{q_1}$.
\end{lem}
\begin{proof}
   This follows for bounded domains from \cite[Theorem 4.6]{Vyb06}; in our case of $\R^d$ the same technique applies as mentioned in \cite[Proposition 1.20]{Tri19} or \cite[Proposition 4.4]{NGS17a}. 
\end{proof} 
We will now introduce symbol classes for pseudodifferential operators. We will use the notation $\langle\xi\rangle:=(1+|\xi|^2)^{\frac12}$ as well as $\ddd(\xi):=\frac{1}{(2\pi)^d}\dd\xi$ and $\ddd(\tau,\xi):=\frac{1}{(2\pi)^{d+1}}\dd(\tau,\xi)$.
\begin{defn}\label{defn:symbolclass}
\begin{enumerate}
    \item\label{defn:symbolclassi} Let $d\in \N$, and $\me=(m_\tau,m_\xi)\in \R^2$.
    We say that $p:\R^{d}\times \R^{d+1}\to \C$ belongs to $S^{\me}_{1,0}$ if $p=p(x,\tau,\xi)$ is smooth in $x\in \R^{d}$, $\tau\in \R$ and $\xi\in \R^d$, and for each $N\in \N$ there holds
\begin{align*}
    \|p\|_{S^\me_{1,0},N}:=\sup \langle\tau\rangle^{-m_\tau+\beta_\tau}\langle\xi\rangle^{-m_\xi+|\beta_\xi|}|\partial_x^\alpha\partial_\tau^{\beta_\tau}\partial_\xi^{\beta_\xi}p(x,\tau,\xi)|<\infty.
\end{align*}
where the supremum runs over all $x\in \R^{d}$, $(\tau,\xi)\in \R^{d+1}$ and all $\alpha\in \N_0^{d}$, $(\beta_\tau,\beta_\xi)\in \N_0^{d+1}$ with $|\alpha|,|\beta_\tau|,|\beta_\xi|\le N$.
For $p\in S^\me_{1,0}$ and $l,j\in \N_0$ we write $p_{lj}(x,\tau,\xi):=\psi_l(\tau)\phi_j(\xi)p(x,\tau,\xi)$.
    \item\label{defn:symbolclassii} For $p\in S^{\me}_{1,0}$ and $u\in \cals(\R^{d+1})$ we introduce
\begin{align*}
    p(x,D_{t,x})u(t,x)&:=\int_{\R^{d+1}}\int_{\R^{d+1}}e^{i(t-s)\tau+i(x-y)\cdot\xi} p(x,\tau,\xi)u(s,y) \dd(s,y) \ddd(\tau,\xi), \\
    p(D_{t,x},x)u(t,x)&:=\int_{\R^{d+1}}\int_{\R^{d+1}}e^{i(t-s)\tau+i(x-y)\cdot\xi} p(y,\tau,\xi)u(s,y) \dd(s,y) \ddd(\tau,\xi).
\end{align*}
We call $p(x,D_{t,x})$ a pseudodifferential operator of order $\me$ in $x$-form, and $p(D_{t,x},x)$ a pseudodifferential operator of order $\me$ in $y$-form.
If $p\in S^\me_{1,0}$ does not depend on $\tau$, we also write $p(x,D_{x})$ and $p(D_{x},x)$, respectively.
\end{enumerate}
\end{defn}
Equipped with the semi-norms $\|p\|_{S^\me_{1,0},N}$, the space $S^\me_{1,0}$ turns into a Fr\'echet space.
Observe that for $p\in S^\me_{1,0}$ and $l,j\in\N_0$ we have
\begin{align*}
 p_{lj}(x,D_{t,x})&=p(x,D_{t,x})\psi_l(D_t)\phi_j(D_x) =\psi_l(D_t)p(x,D_{t,x})\phi_j(D_x), \\
 p_{lj}(D_{t,x},x)&=\psi_l(D_t)\phi_j(D_x)p(D_{t,x},x) =\phi_j(D_x)p(D_{t,x},x)\psi_l(D_t).
\end{align*}
We associate to $p\in S^{\me}_{1,0}$ the symbols
\begin{align*}
    p_L(x,\tau,\xi):=\int_{\R^d}\int_{\R^d}e^{-iy\cdot\eta} p(x+y,\tau,\xi+\eta)\dd y \ddd\eta, \\
    p_R(x,\tau,\xi):=\int_{\R^d}\int_{\R^d}e^{iy\cdot\eta} p(x+y,\tau,\xi+\eta)\dd y \ddd\eta,
\end{align*}
where the integrals are understood in the oscillatory sense.
Moreover, for $r\in S^{\me_1}_{1,0}$ and $q\in S^{\me_2}_{1,0}$ we introduce the oscillatory integral
\begin{align*}
    (r\# q)(x,\tau,\xi):=\int_{\R^{d}}\int_{\R^{d}} e^{-iy\cdot \eta} r(x,\tau,\xi+\eta)q(x+y,\tau,\xi)\dd y \ddd\eta.
\end{align*}
\begin{lem}\label{lem:pseudo_comp}
Let $\me,\me_1,\me_2\in \R^2$ and $p\in S^{\me}_{1,0}$, $r\in S^{\me_1}_{1,0}$ and $q\in S^{\me_2}_{1,0}$.
Then for every $x\in \R^d$ and $(\tau,\xi)\in\R^{d+1}$, the above oscillatory integrals are well-defined, and there holds $p_L, p_R\in S^{\me}_{1,0}$ as well as $r\# q\in S^{\me_1+\me_2}_{1,0}$.
The mappings $p\mapsto p_L$, $p\mapsto p_R$ are bounded linear mappings on $S^{\me}_{1,0}$, while $(r,q)\mapsto r\# q$ is a bounded bilinear mapping from $S^{\me_1}_{1,0}\times S^{\me_2}_{1,0}$ to $S^{\me_1+\me_2}_{1,0}$.
More precisely, for $N\in\N$ and $k\in \N_0$ there is a $c>0$ such that
\begin{align*}
\left\|r\# q - \sum_{|\alpha|< k} \frac{1}{\alpha!} (\partial_\xi^\alpha r)(D_x^\alpha q)\right\|_{S^{\me_1+\me_2-k};N} \le c \sum_{|\alpha|=k}\|\partial_\xi^\alpha r\|_{S^{\me_1}_{1,0};N}\|D_x^\alpha q\|_{S^{\me_2}_{1,0};N}.
\end{align*}
Finally, for $u\in \cals(\R^{d+1})$ there holds
\begin{align*}
    p_L(x,D_{t,x})u=p(D_{t,x},x)u, \quad p_R(D_{t,x},x)u=p(x,D_{t,x})u, \quad (r\# q)(x,D_{t,x})=r(x,D_{t,x})q(x,D_{t,x})u.
\end{align*}
\end{lem}
\begin{proof}
    Analogous to Theorem 3.16 and Theorem 3.32 in \cite{Abe12}. The boundedness of the mappings follows by an inspection of the proof of Theorem 3.15 in \cite{Abe12}.
\end{proof}
In particular, Lemma \ref{lem:pseudo_comp} states that the composition $r(x,D_{t,x})q(x,D_{t,x})$ is a pseudodifferential operator of order $\me_1+\me_2$ in $x$-form, and its symbol is given by $q\# p\in S^{\me_1+\me_2}_{1,0}$.

We record the following continuity result for symbols of pseudodifferential operators well-adapted to our Averaging Lemma \ref{lem:av}.
We recall the definition of the symbol class $S^{\me}_{1,0}$ in Definition \ref{defn:symbolclass}.

\begin{thm}\label{thm:FM}
  Let $r\in S^{\me}_{1,0}$ with $\me:=(m_\tau,m_\xi)\in \R^2$.
  Then there is $N=N(d)\in\N$ such that for all $p\in[1,\infty]$ there is $C>0$ such that for all $l,j\in \N_0$ the operators $r_{lj}(x,D_{t,x})$ and $r_{lj}(D_{t,x},x)$ extend to bounded linear operators on $\LR{p}(\R^{d+1})$ with
  \begin{align*}
  \|r_{lj}(x,D_{t,x})\|_{\LR{p}_{t,x}\to\LR{p}_{t,x}},\|r_{lj}(D_{t,x},x)\|_{\LR{p}_{t,x}\to\LR{p}_{t,x}}\le C2^{m_\tau l + m_\xi j}\|r\|_{S^{\me}_{1,0};N}.
  \end{align*}
  Furthermore, the same is true if we replace $\LR{p}(\R^{d+1})$ by $\calm_{TV}$.
\end{thm}
\begin{proof}
Follows as in Lemma 6.20 and Remark 6.21 in \cite{Abe12}. The dependence on $\|r\|_{S^{\me}_{1,0};N}$ in the bound can be seen by inspecting the proofs, cf.~Remark 5.21 in \cite{Abe12}.
\end{proof}

\section{Optimality of Estimates via Scaling}\label{Scale}
The following lemma based on a scaling argument extends the local case \cite[Lemma 3.1]{GST20} to the case of a fractional porous medium equation. Since the proof is virtually the same as in \cite{GST20}, we omit it here.

\begin{lem}\label{lem:scaling}
 Let $T>0$, $m>1$, $\frex\in (0,2)$, $\mu\in [1,m]$, $p\in [1,\infty)$ and $\sigma_t,\sigma_x \ge 0$. Assume that there is a constant $c=c(d,\frex,m,\mu,p,\sigma_t,\sigma_x)>0$ such that
\begin{align}\label{scale_pme_lp_est}
 \norm{\uf^{[\mu]}}_{\DSR{\sigma_t}{p}(0,T;\DSR{\sigma_x}{p}(\R^d))}^p\le c \vpp{\norm{\uf_0}_{\LR{1}(\R^d)} + \norm{S}_{\LR{1}(0,T;\LR{1}(\R^d))}}
\end{align}
for all solutions $\uf$ to \eqref{pme_sys}. Then, necessarily, 
\begin{align}\label{scaling_const}
\begin{split}
 p &\le \frac{m}{\mu+(m-1)\sigma_t}\le \frac{m}{\mu}, \\
 \sigma_t&\le \frac{m-\mu p}{p(m-1)}\le \frac{m-\mu}{m-1}, \quad \tand\\
 \sigma_x &= \frac{\mu p -1}{p}\frac{\frex}{m-1}\le \frac{\frex (\mu-\sigma_t)}{m}\le \frac{\frex \mu}{m}. 
\end{split}
\end{align}
In particular, if $\sigma_t=\frac{m-\mu}{m-1}$, then $p=1$ and $\sigma_x= \frac{\frex (\mu-1)}{m-1}$.
\end{lem}
\begin{rem}
 As in the local case, Lemma \ref{lem:scaling} shows that in the whole space, the regularity exponent $\sigma_x\in [\frac{\frex(\mu-1)}{m-1},\frac{\frex\mu}{m}]$ is in a one-to-one correspondence to the integrability exponent $p\in[1,\frac{m}{\mu}]$ via
 \begin{align*}
  \sigma_x= \frac{\mu p -1}{p}\frac{\frex}{m-1}, \quad \tand \quad p= \frac{\frex}{\frex\mu-\sigma_x(m-1)}.
 \end{align*}
\end{rem}


For every $m\in (1,\infty)$ and $\frex\in(0,2)$ there is suitable bounded and H\"older continuous function $f:[0,\infty)\to\R$, such that 
\begin{align}\label{def:ubb}
 \ubb(t,x):=t^{-\alpha}f(|x|t^{-\beta}),
\end{align}
with $\alpha:=\frac{d}{d(m-1)+\frex}$ and $\beta=\frac{\alpha}{d}$ is a self-similar solution to \eqref{pme_sys}, see \cite{Vaz14}.
In the local case, $\ubb$ is called the Barenblatt solution, and we continue to use this name in our non-local situation.
\begin{prop}
 Let $m\in (1,\infty)$ and $\frex\in(0,2)$, and let $\ubb$ be given by \eqref{def:ubb}.
 Then, for $\mu\in [1,m]$,
%
 $\displaystyle \ubb^{[\mu]}\in L^{\frac{m}{\mu}}(0,T;\DSR{\sigma}{\frac{m}{\mu}}(\R^d))$ 
%
implies $\sigma<\frac{\frex\mu}{m}$.
\end{prop}
\begin{proof}
 With $F(x):=f(x)^\mu$ we have $\ubb^{[\mu]}(t,x)=t^{-\alpha\mu}F(xt^{-\beta})$. We next observe that, for $\sigma\in (0,1)$ and each $t\ge 0$,
 \begin{align*}
  \norm{\ubb^{[\mu]}(t,\cdot)}_{\DSR{\sigma}{\frac{m}{\mu}}(\R^d)}^{\frac{m}{\mu}} & =\int_{\R^d\times\R^d} \frac{|\ubb^{[\mu]}(t,x)-\ubb^{[\mu]}(t,y)|^{\frac{m}{\mu}}}{|x-y|^{\frac{\sigma m}{\mu}+d}}\dd x \dd y \\
  &=t^{-\alpha m - \beta(\frac{\sigma m}{\mu}+d) + 2d\beta} \norm{F}_{\DSR{\sigma}{\frac{m}{\mu}}(\R^d)}^{\frac{m}{\mu}}.
 \end{align*}
 Hence,
 \begin{align*}
  \norm{\ubb^{[\mu]}}_{L^{\frac{m}{\mu}}(0,T;\DSR{\sigma}{\frac{\mu}{m}}(\R^d))}^{\frac{m}{\mu}} =\norm{t^{-\alpha m - \beta(\frac{\sigma m}{\mu}+d) + 2d\beta}}_{L^1(0,T)} \norm{F}_{\DSR{\sigma}{\frac{m}{\mu}}(\R^d)}^{\frac{m}{\mu}},
 \end{align*}
 which is finite if and only if
 \begin{align*}
  -\alpha m - \beta(\frac{\sigma m}{\mu}+d) + 2d\beta > -1 \quad \tand \quad F\in \DSR{\sigma}{\frac{m}{\mu}}(\R^d).
 \end{align*}
 Hence, necessarily
 \begin{align*}
  m + \frac1d(\frac{\sigma m}{\mu}+d) - 2 < \frac{1}{\alpha} = \frac{d(m-1)+\frex}{d},
 \end{align*}
 which is equivalent to $\sigma <\frac{\frex\mu}{m}$. In the case $\sigma\in (1,2)$ we observe that it holds $\partial_{x_i}\ubb^{[\mu]}(t,x)=t^{-\alpha\mu+\beta}\partial_{x_i}F(xt^{-\beta})$, so that analogous arguments may be applied.
\end{proof}

\section{Averaging Lemmata}\label{AvLem}
For $y\in\R^{d}$ let $h_1(\xi,y):=e^{iy\cdot\xi}-1$ and $h_2(\xi,y):=e^{iy\cdot\xi}-1-i\xi\cdot y$. The formulae
\begin{align*}
    u(x+y)-u(x)&=\calf^{-1}[h_1(\cdot,y)\calf u](x), \\
    u(x+y)-u(x)-y\cdot\nabla u(x)&=\calf^{-1}[h_2(\cdot,y)\calf u](x),
\end{align*}
show that the operator $L$ given in \eqref{def_op} has a representation as a pseudodifferential operator in $x$-form, \ie
\begin{align}\label{fourier-repr-op}
Lu(x)=p(x,D_x)u(x):=\int_{\R^d} e^{ix\cdot\xi} p(x,\xi) [\calf u](\xi) \ddd \xi, \qquad x\in\R^d,
\end{align}
where 
\begin{align}\label{fourier-repr-ker}
p(x,\xi):=-\int_{\R^d} (e^{iy\cdot\xi}-1-\1_{[1,2)}(a)i\xi\cdot y)k(x,y)\dd y, \qquad x\in \R^d, \ \xi\in\R^d,
\end{align}
Observe that there are different sign conventions for the operator $L$, cf.~\cite{AbK09,AAO20}, and we opt for the one of \cite{AAO20}.
\begin{lem}\label{lem:symbol_est}
    Let Assumption \ref{ass_1} be valid for some $\frex\in(0,2)$ and $m\in(1,\infty)$, and let $p$ be given by \eqref{fourier-repr-ker}. 
    \begin{enumerate}
        \item\label{lem:symbol_esti} There holds $p\in S^{(0,\frex)}_{1,0}$.
        \item\label{lem:symbol_estii} There is $c>0$ such that
        \begin{align*}
         \Re p(x,\xi)\ge c|\xi|^{\frex}, \qquad \forall x,\xi\in\R^d.
        \end{align*}
       In particular there exists $M>0$ with
       \begin{align*}
           \left|\frac{\Im p(x,\xi)}{\Re p(x,\xi)}\right|\le M , \qquad \forall x,\xi\in\R^d.
       \end{align*}
        \item\label{lem:symbol_estiii} For $v,\tau\in\R$, $x,\xi\in\R^d$ define
        \begin{align}\label{av_op}
         \diffop_v(x,\tau,\xi) &:= i\tau + \nonl'(v)p(x,\xi).
        \end{align}
        Then there exists $c>0$ such that we have
        \begin{align*}
            |\call_v(x,\tau,\xi)|\ge c(|\tau|+\nonl'(v)|\xi|^\frex), \qquad \forall v,\tau\in\R, \ x,\xi\in\R^d.
        \end{align*}
    \end{enumerate}
\end{lem}
\begin{proof}
    We provide a proof along the lines of \cite[Lemma 2.13]{AbK09}.
    \begin{enumerate}
        \item Observe that the function $h:\R\to\C$, $h(s):=e^{is}-1-\1_{[1,2)}(a)is$ admits the bound
    \begin{align*}
        h(s)\le C\frac{|s|^j}{1+|s|}, \quad s\in\R,
    \end{align*}
    where $j=1$ if $\frex\in (0,1)$ and $j=2$ if $\frex\in[1,2)$.
    We first give the proof in the case $\frex\ne 1$.
    Let $N\in \N$ and $\alpha,\beta,\gamma\in \N_0^d$ with $|\alpha|\le N$ and $|\beta|=|\gamma|=\ell\le N$.
    Moreover, let $\eta\in \CRci(\R^d)$ be such that $0\le \eta\le 1$, $\eta(y)=1$ for $y\in B_1(0)$, and $\supp\eta\subset B_2(0)$. For $R>0$ define $\eta_R(y):=\eta(R^{-1}y)$. 
    Then for all $x\in\R^d$ and $\xi\in\R^d\setminus\set{0}$ we have
    \begin{align}\label{symb_est_1}
    \begin{split}
        -\int_{\R^d} \partial_\xi^\beta &(\xi^\gamma h(\xi\cdot y)) \eta_R(y) \partial_x^\alpha k(x,y)\dd y =-\int_{\R^d} \partial_y^\gamma (y^\beta h(\xi\cdot y)) \eta_R(y) \partial_x^\alpha k(x,y)\dd y \\
        &=(-1)^{\ell+1} \sum_{\gamma_1+\gamma_2=\gamma}\int_{\R^d} y^\beta h(\xi\cdot y) \partial_y^{\gamma_1}\eta_R(y)\partial_y^{\gamma_2} \partial_x^\alpha k(x,y)\dd y.
    \end{split}
    \end{align}
    Whenever $\gamma_1+\gamma_2=\gamma$ and $\gamma_1\ne 0$, we have that $\partial_y^\gamma\eta$ is supported on the closure of $B_2(0)\setminus B_1(0)$, and thus \eqref{ass_ker_1} yields
    \begin{align*}
       &\left|\int_{\R^d} y^\beta h(\xi\cdot y) \partial_y^{\gamma_1}\eta_R(y)\partial_y^{\gamma_2} \partial_x^\alpha k(x,y)\dd y\right| \\
       &\quad \le C\int_{B_{2R}(0)\setminus B_R(0)} |y|^\ell |y|^{j-1} R^{-|\gamma_1|}|\partial_y^{\gamma_1}\eta(R^{-1}y)||y|^{-d-\frex-|\gamma_2|}\dd y \\
       &\quad = C R^{j-1-\frex}\int_{B_2(0)\setminus B_1(0)}|z|^{\ell+j-1-d-\frex-|\gamma_2|}\eta(z) \dd z \to 0 \quad \text{as } R\to \infty.  
    \end{align*}
    For the term $\gamma_1+\gamma_2=\gamma$ in \eqref{symb_est_1} with $\gamma_1=0$ we have by dominated convergence
    \begin{align*}
        (-1)^{\ell+1} \int_{\R^d} y^\beta & h(\xi\cdot y) \eta_R(y)\partial_y^{\gamma} \partial_x^\alpha k(x,y)\dd y \to (-1)^{\ell+1} \int_{\R^d} y^\beta h(\xi\cdot y) \partial_y^{\gamma} \partial_x^\alpha k(x,y)\dd y \\
        &=(-1)^{\ell+1} |\xi|^{-d-\ell}\int_{\R^d} y^\beta h(\frac{\xi}{|\xi|}\cdot y) \partial_y^\gamma \partial_x^\alpha k(x,\frac{y}{|\xi|})\dd y.
    \end{align*}
    Since the left-hand side in \eqref{symb_est_1} tends to $\partial_x^\alpha \partial_\xi^\beta\xi^\gamma p(x,\xi)$, we have
    \begin{align*}
        |\partial_x^\alpha\partial_\xi^\beta \xi^\gamma p(x,\xi)|\le C|\xi|^{-d-\ell} \int_{\R^d} |y|^\ell\frac{|y|^{j}}{1+|y|} \left|\frac{y}{|\xi|}\right|^{-d-\frex-\ell} \dd y = C' |\xi|^{\frex}.
    \end{align*}
    By induction on $\ell$ we thus infer
    \begin{align*}
        |\xi^\gamma\partial_x^\alpha \partial_\xi^\beta p(x,\xi)|\le C' |\xi|^{\frex}
    \end{align*}
    for all $\alpha,\beta,\gamma\in\N_0^d$ with $|\alpha|\le N$ and $|\beta|=|\gamma|=\ell\le N$.
    This is the claimed estimate for $\frex\ne 1$.
    
    \medskip
    
    In the case $\frex=1$ we have the additional assumption $k(x,y)=k(x,-y)$, which in particular implies by a similar cut-off argument as above
     \begin{align*}
        \partial_x^\alpha \partial_\xi^\beta \xi^\gamma p(x,\xi)&= (-1)^{|\beta|+1} |\xi|^{-d-\ell}\int_{|y|\ge 1} y^\beta h(\frac{\xi}{|\xi|}\cdot y) \partial_y^\gamma \partial_x^\alpha k(x,\frac{y}{|\xi|})\dd y \\
        &\quad + (-1)^{|\beta|+1} |\xi|^{-d-\ell}\int_{|y|<1} y^\beta \left(e^{i\frac{\xi}{|\xi|}\cdot y}-1\right) \partial_y^\gamma \partial_x^\alpha k(x,\frac{y}{|\xi|})\dd y.
    \end{align*}
   Therefore we obtain also in the case $\frex=1$ the estimate
    \begin{align*}
        |\partial_x^\alpha \partial_\xi^\beta \xi^\gamma p(x,\xi)|\le C|\xi|^{-d-\ell} \int_{\R^d} |y|^\ell\frac{|y|^{2}}{1+|y|^2} \left|\frac{y}{|\xi|}\right|^{-d-1-\ell} \dd y = C' |\xi|.
    \end{align*}
    The rest of the proof is identical to the case $\frex\ne 1$.
    \item By \eqref{ass_ker_2} and the change of variables formula we have for all $|\xi|>0$
    \begin{align*}
        \Re p(x,\xi)&=\int_{\R^d} (1-\cos(\xi\cdot y))k(x,y)\dd y\ge c\int_{B_{1/|\xi|}(0)} (1-\cos(\xi\cdot y))|y|^{-d-\frex}\dd y \\
        &=c|\xi|^\frex\int_{B_{1}(0)} (1-\cos(\frac{\xi}{|\xi|}\cdot z))|z|^{-d-\frex}\dd z \ge c'|\xi|^\frex,
    \end{align*}
    where we used that the value of the integral $\int_{B_{1}(0)} (1-\cos(\omega\cdot z)|z|^{-d-\frex}\dd z$ is independent of the unit vector $\omega$.
   The assertion about the quotient of the imaginary and real part now follows by combining the upper with the lower bound.
    \item 
   Since $\frac{|\Im p(x,\xi)|}{\Re p(x,\xi)}\le M$ for $|\xi|>0$, the symbol $p(x,\xi)$ is contained in a sector of the complex plane of angle $\arctan M<\frac\pi 2$, and thus so is $\nonl'(v)p(x,\xi)$ due to $\nonl'(v)\ge 0$.
   This sector does not contain the imaginary axis, and consequently the complex number $\call_v(x,\tau,\xi)=i\tau+\nonl'(v)p(x,\xi)$ is at least of comparable size to both $i\tau$ and $\nonl'(v)p(x,\xi)$.
    \end{enumerate}
\end{proof}
The next lemma shows that for all $v\in \R\setminus\set{0}$ the operators $\call_v(x,D_{t,x})$ and $\nonl'(v)p(x,D_x)$ possess a parametrix, i.e.~they can be inverted up to a lower order error term, with precise bounds in terms of $v$ on the corresponding symbols.
We note that due to Lemma \ref{lem:pseudo_comp} the parametrices and their errors can be written in $y$-form, a fact that we will exploit heavily in Lemma \ref{lem:av} and Lemma \ref{lem:av3} below.
The idea of finding the parametrices is standard (see for example \cite{Abe12}).
Since we will have to be precise about the $v$-dependence of the bounds, we include a full argument here. 
\begin{lem}\label{lem:FM}
 Suppose  that Assumption \ref{ass_1} is valid for some $\frex\in(0,2)$ and $m\in(1,\infty)$.
 Define $p(x,\xi)$ via \eqref{fourier-repr-ker}.
 Let $\call_v(x,\tau,\xi):=i\tau+\nonl'(v)p(x,\xi)$. Then for each $v\in \R$ with $\nonl'(v)\ne 0$ there are symbols $\calm_v,\calr_v\in S^{(0,0)}_{1,0}$ (explicitly given by \eqref{lem:FMe001} below) such that 
 \begin{align*}
  \calm_v(D_{t,x},x)\call_v(x,D_{t,x})=\id - \calr_v(D_{t,x},x)   
 \end{align*}
 and subject to the following additional smoothing properties:
 For each $N\in\N$ there is $c>0$ such that
 \begin{align*}
     \|\calm_v\|_{S^{(0,-\frex)}_{1,0;N}}\le c\nonl'(v)^{-1}, \quad
     \|\partial_v\calm_v\|_{S^{(0,-\frex)}_{1,0;N}}\le c\frac{|\nonl''(v)|}{\nonl'(v)^{2}}, \quad
     \|\calr_v\|_{S^{(0,-\frex)}_{1,0;N}}\le c,
 \end{align*}
 and
 \begin{align*}
     \|\calm_v\|_{S^{(-1,0)}_{1,0;N}}\le c, \quad
     \|\partial_v\calm_v\|_{S^{(-2,\frex)}_{1,0;N}}\le c|\nonl''(v)|, \quad
     \|\calr_v\|_{S^{(-1,0)}_{1,0;N}}\le c(1+\nonl'(v)),
 \end{align*}
 where the bounds on $\partial_v\calm_v$ are true for almost all $v\in\R$.
 Moreover, there are symbols $\td\calm_v,\td\calr_v\in S^{(0,0)}_{1,0}$ independent of $\tau$ such that
 \begin{align*}
  \td\calm_v(D_x,x)\nonl'(v)p(x,D_x)=\id - \calr_v(D_x,x)    
 \end{align*}
 and such that for each $N\in\N^2$ there is $c>0$ with
 \begin{align*}
     \|\td\calm_{v}\|_{S^{(0,-\frex)}_{1,0;N}}\le c \nonl'(v)^{-1}, \quad
     \|\partial_v \td\calm_{v}\|_{\CR{\he}S^{(0,-\frex)}_{1,0;N}}\le c \frac{|\nonl''(v)|}{\nonl'(v)^2}, \quad
     \|\td\calr_v\|_{S^{(0,-\frex)}_{0,1};N}\le c.
 \end{align*}
\end{lem}
\begin{proof}
Let $v\in \R\setminus\set{0}$ be such that $\nonl'(v)\ne 0$.
Define the following symbols:
 \begin{align}
 \begin{split}\label{lem:FMe001}
     \calm_v'(x,\tau,\xi)&:=\frac{1-\psi_0(\tau)-\phi_0(\xi)+\psi_0(\tau)\phi_0(\xi)}{\call_v(x,\tau,\xi)}, \\
     \calr_{1,v}(x,\tau,\xi)&:=(\calm_v'\# \call_v)(x,\tau,\xi)-(\calm_v'\call_v)(x,\tau,\xi), \\
     \calr_2(\tau,\xi)&:=(\calm_v'\call_v)(x,\tau,\xi)-1=\psi_0(\tau)+\phi_0(\xi)-\psi_0(\tau)\phi_0(\xi),\\
     \calr_v'(x,\tau,\xi)&:=\calr_{1,v}(x,\tau,\xi)+\calr_2(\tau,\xi), \\
     \calm_{v}''(x,\tau,\xi)&:=\calm_v'(x,\tau,\xi) - (\calr_v'\#\calm_v')(x,\tau,\xi), \\
     \calr_v''(x,\tau,\xi)&:=(\calr_v'\#\calr_v')(x,\tau,\xi), \\
     \calm_v(x,\tau,\xi)&:=(\calm_v'')_R(x,\tau,\xi), \\
     \calr_v(x,\tau,\xi)&:=(\calr_v'')_R(x,\tau,\xi).
 \end{split}
 \end{align}
By Lemma \ref{lem:symbol_est} we have for $\alpha,\beta\in\N_0^d$ which do not both vanish
 \begin{align*}
  \|\partial_x^\beta\partial_\xi^{\alpha}\call_v(\cdot,\tau,\xi)\|_{\LR{\infty}(\R^d)} &= \|\nonl'(v)\partial_x^\beta\partial_\xi^{\alpha}p(\cdot,\xi)\|_{\LR{\infty}(\R^d)}\lesssim \nonl'(v)|\xi|^{\frex-\alpha}, 
 \end{align*}
 and
 \begin{align*}
     |\call_v(x,\tau,\xi)|=|i\tau+\nonl'(v)p(x,\xi)|\gtrsim |\tau| + \nonl'(v)|\xi|^\frex,
     \qquad \tau\in\R, \ x,\xi\in\R^d.
 \end{align*}
 Here and in the rest of the proof, the implicit constant implied by the notation $\lesssim$ may depend on $N\in \N$, but is independent of $v$.
 Let $s\in[0,1]$.
 Then
 \begin{align}\label{eq:faadibruno_e1}
     \left\|\frac1{\call_v(\cdot,\tau,\xi)}\right\|_{\LR{\infty}(\R^d)} \lesssim
     |\tau|^{s-1}\nonl'(v)^{-s}|\xi|^{-s\frex}.
 \end{align}
 Analogously, $\partial_v\call(x,\tau,\xi)=\nonl''(v)p(x,\xi)$ implies
 \begin{align}\label{eq:faadibruno_e2}
     \left\|\partial_v\frac1{\call_v(\cdot,\tau,\xi)}\right\|_{\LR{\infty}(\R^d)} \lesssim
     |\tau|^{2(s-1)}|\nonl''(v)|\nonl'(v)^{-2s}|\xi|^{(1-2s)\frex}.
 \end{align}
 We recall from the multivariate formula of Fa\`{a} di Bruno \cite{Har06} that we have for all $0\ne \alpha\in\N_0^d$
 \begin{align}\label{eq:faadibruno}
  \partial_\xi^{\alpha} \frac{1}{f(\xi)}&=\frac{c_\alpha}{f(\xi)}\sum_{\beta_1+\ldots +\beta_{|\alpha|}=\alpha}\prod_{j=1}^{|\alpha|}\frac{\partial_\xi^{\beta_j}f(\xi)}{f(\xi)}.
 \end{align}
 Applied to the function $f(\xi):=\call_v(x,\tau,\xi)$ we obtain from the formulae \eqref{eq:faadibruno} and \eqref{eq:faadibruno_e1}, and from $\partial_\tau\call_v(x,\tau,\xi)=i$, that for $|\tau|\ge 1$ and $|\xi|\ge 1$ there holds
 \begin{align*}
  \left\|\partial_x^\beta\partial_\tau^{\alpha_\tau}\partial_\xi^{\alpha_\xi} \frac1{\call_v(\cdot,\tau,\xi)}\right\|_{\LR{\infty}(\R^d)}&\lesssim \left\|\frac1{\call_v(\cdot,\tau,\xi)}\right\|_{\LR{\infty}(\R^d)}|\tau|^{-|\alpha_\tau|}|\xi|^{-|\alpha_\xi|} \\
  &\lesssim \nonl'(v)^{-s}\langle\tau\rangle^{s-1-|\alpha_\tau|}\langle \xi\rangle^{-s\frex-|\alpha_\xi|}.
 \end{align*}
 Since $\calm_v'=\frac{1-\psi_0-\phi_0+\psi_0\phi_0}{\call_v}$, we have $\calm_v'(x,\tau,\xi)=0$ for all $|\tau|\le 1$ or $|\xi|\le 1$, so that
 \begin{align*}
     \|\calm_v'\|_{S^{(s-1,-s\frex)}_{1,0;N}}\lesssim \nonl'(v)^{-s}.
 \end{align*}
 Applying $\partial_v$ to \eqref{eq:faadibruno} with $f(\xi):=\call_v(x,\tau,\xi)$, we also obtain from the product rule and \eqref{eq:faadibruno_e2}, that
 \begin{align*}
     \|\partial_v\calm_v'\|_{S^{(2(s-1),(1-2s)\frex)}_{1,0;N}}\lesssim |\nonl''(v)|\nonl'(v)^{-2s}.
 \end{align*}
Hence, by Lemma \ref{lem:pseudo_comp} we have
\begin{align*}
    \|\calr_{1,v}\|_{S^{(0,-1)}_{1,0;N}}&\lesssim \sum_{|\alpha|=1}\|\partial_\xi^\alpha \calm_v'\|_{S^{(0,-\frex-1)}_{1,0;N}}\|D_x^\alpha\call_v\|_{S^{(0,\frex)}_{1,0;N}}\le \sum_{|\alpha|=1}\|\calm_v'\|_{S^{(0,-\frex)}_{1,0;N}}\|D_x^\alpha\call_v\|_{S^{(0,\frex)}_{1,0;N}}\lesssim 1, \\
    \|\calr_{1,v}\|_{S^{(-1,\frex-1)}_{1,0;N}}&\le \sum_{|\alpha|=1}\|\partial_\xi^\alpha\calm_v'\|_{S^{(-1,-1)}_{1,0;N}}\|D_x^\alpha\call_v\|_{S^{(0,\frex)}_{1,0;N}} \le \sum_{|\alpha|=1}\|\calm_v'\|_{S^{(-1,0)}_{1,0;N}}\|D_x^\alpha\call_v\|_{S^{(0,\frex)}_{1,0;N}}\lesssim \nonl'(v),
\end{align*}
as well as
\begin{align*}
    \|\partial_v\calr_{1,v}\|_{S^{(0,-1)}_{1,0;N}}&\le \sum_{|\alpha|=1}\left(\|\partial_\xi^\alpha \partial_v\calm_v'\|_{S^{(0,-\frex-1)}_{1,0;N}}\|D_x^\alpha\call_v\|_{S^{(0,\frex)}_{1,0;N}} + \|\partial_\xi^\alpha \calm_v'\|_{S^{(0,-\frex-1)}_{1,0;N}}\|D_x^\alpha\partial_v\call_v\|_{S^{(0,\frex)}_{1,0;N}}\right)\\
    &\le\sum_{|\alpha|=1}\left(\|\partial_v\calm_v'\|_{S^{(0,-\frex)}_{1,0;N}}\|D_x^\alpha\call_v\|_{S^{(0,\frex)}_{1,0;N}} + \|\calm_v'\|_{S^{(0,-\frex)}_{1,0;N}}\|D_x^\alpha\partial_v\call_v\|_{S^{(0,\frex)}_{1,0;N}}\right)\\
    &\lesssim \left(\frac{|\nonl''(v)|}{\nonl'(v)^2}\nonl'(v) + \frac{1}{\nonl'(v)}|\nonl''(v)|\right)\lesssim \frac{|\nonl''(v)|}{\nonl'(v)},\\
    \|\partial_v\calr_{1,v}\|_{S^{(-1,\frex-1)}_{1,0;N}}
    &\lesssim \sum_{|\alpha|=1}\left(\|\partial_\xi^\alpha \partial_v\calm_v'\|_{S^{(-1,-1)}_{1,0;N}}\|D_x^\alpha\call_v\|_{S^{(0,\frex)}_{1,0;N}} + \|\partial_\xi^\alpha \calm_v'\|_{S^{(-1,-1)}_{1,0;N}}\|D_x^\alpha\partial_v\call_v\|_{S^{(0,\frex)}_{1,0;N}}\right)\\
    &\le \sum_{|\alpha|=1}\left(\|\partial_v\calm_v'\|_{S^{(-1,0)}_{1,0;N}}\|D_x^\alpha\call_v\|_{S^{(0,\frex)}_{1,0;N}} + \|\calm_v'\|_{S^{(-1,0)}_{1,0;N}}\|D_x^\alpha\partial_v\call_v\|_{S^{(0,\frex)}_{1,0;N}}\right)\\
    &\lesssim \left(\frac{|\nonl''(v)|}{\nonl'(v)}\nonl'(v) + |\nonl''(v)|\right)\lesssim |\nonl''(v)|.
\end{align*}
From $\frex\le 2$ we also infer
\begin{align*}
    \|\calr_{1,v}\|_{S^{(-\frac12,0)}_{1,0;N}}\le \|\calr_{1,v}\|_{S^{(-\frac12,\frac{\frex}{2}-1)}_{1,0;N}}\le \|\calr_{1,v}\|_{S^{(-1,\frex-1)}_{1,0;N}}^\frac12\|\calr_{1,v}\|_{S^{(0,-1)}_{1,0;N}}^\frac12 \lesssim \nonl'(v)^\frac12.
\end{align*}
Since $\calr_{2}(x,\tau,\xi)=\psi_0(\tau)+\phi_0(\xi)-\psi_0(\tau)\phi_0(\xi)$ is compactly supported (and independent of $v\in \R$), it holds $\|\calr_2\|_{S^{(0,-1)}_{1,0;N}} + \|\calr_2\|_{S^{(-\frac12,0)}_{1,0;N}}\lesssim 1$.
Together, we obtain for $\calr_v'=\calr_{1,v}+\calr_2$ that
\begin{align*}
\|\calr_v'\|_{S^{(0,-1)}_{1,0;N}}&\lesssim 1, \quad \|\calr_v'\|_{S^{(-\frac12,0)}_{1,0;N}}\lesssim 1+\nonl'(v)^\frac12, \\
\|\partial_v\calr_v'\|_{S^{(0,-1)}_{1,0;N}}& \lesssim \frac{|\nonl''(v)|}{\nonl'(v)}, \quad
\|\partial_v\calr_v'\|_{S^{(-1,\frex-1)}_{1,0;N}}\lesssim |\nonl''(v)|.
\end{align*}
Next we use Lemma \ref{lem:pseudo_comp} to obtain
\begin{align*}
    \|\calr_v''\|_{S^{(0,-\frex)}_{1,0;N}}&\le \|\calr_v''\|_{S^{(0,-2)}_{1,0;N}}\le \|\calr_v'\|_{S^{(0,-1)}_{1,0;N}}^2\lesssim 1, \\
    \|\calr_v''\|_{S^{(-1,0)}_{1,0;N}}&\lesssim \|\calr_v'\|_{S^{(-\frac12,0)}_{1,0;N}}^2 \lesssim (1+\nonl'(v)^\frac12)^2\lesssim 1+\nonl'(v),
\end{align*}
as well as
\begin{align*}
    \|\calm_v''\|_{S^{(0,-\frex)}_{1,0;N}}&\le (1+\|\calr_v'\|_{S^{(0,0)}_{1,0;N}})\|\calm_v'\|_{S^{0,-\frex}_{1,0;N}}\lesssim (1+\|\calr_v'\|_{S^{(0,-1)}_{1,0;N}})\|\calm_v'\|_{S^{0,-\frex}_{1,0;N}}\lesssim \nonl'(v)^{-1}, \\
    \|\calm_v''\|_{S^{(-1,0)}_{1,0;N}}&\lesssim (1+\|\calr_v'\|_{S^{(0,0)}_{1,0;N}})\|\calm_v'\|_{S^{(-1,0)}_{1,0;N}}\lesssim (1+\|\calr_v'\|_{S^{(0,-1)}_{1,0;N}})\|\calm_v'\|_{S^{(-1,0)}_{1,0;N}}\lesssim 1,\\
    \|\partial_v\calm_v''\|_{S^{(0,-\frex)}_{1,0;N}}&\le (1+\|\calr_v'\|_{S^{(0,0)}_{1,0;N}})\|\partial_v\calm_v'\|_{S^{(0,-\frex)}_{1,0;N}} + \|\partial_v\calr_v'\|_{S^{(0,0)}_{1,0;N}}\|\calm_v'\|_{S^{(0,-\frex)}_{1,0;N}}\lesssim \frac{|\nonl''(v)|}{\nonl'(v)^{2}},\\
    \|\partial_v\calm_v''\|_{S^{(-2,\frex)}_{1,0;N}}&\le (1+\|\calr_v'\|_{S^{(0,0)}_{1,0;N}})\|\partial_v\calm_v'\|_{S^{(-2,\frex)}_{1,0;N}} + \|\partial_v\calr_v'\|_{S^{(-1,\frex)}_{1,0;N}}\|\calm_v'\|_{S^{(-1,0)}_{1,0;N}} \\
    & \le (1+\|\calr_v'\|_{S^{(0,-1)}_{1,0;N}})\|\partial_v\calm_v'\|_{S^{(-2,\frex)}_{1,0;N}} + \|\partial_v\calr_{1,v}\|_{S^{(-1,\frex-1)}_{1,0;N}}\|\calm_v'\|_{S^{(-1,0)}_{1,0;N}}\lesssim |\nonl''(v)|.
\end{align*}
By Lemma \ref{lem:pseudo_comp}, the estimates on the symbols $\calr_v$, $\calm_v$ and $\partial_v\calm_v$ carry over from those on $\calr_v''$, $\calm_v''$ and $\partial_v\calm_v''$.
By the definitions of $\calm_{v}'$ and $\calr_{v}'$ we have
\begin{align*}
    \calm_{v}'(x,D_{t,x})\call_v(x,D_{t,x})=\id + \calr_{v}'(x,D_{t,x}).
\end{align*}
and further
\begin{align*}
    \calm_v(D_{t,x},x)\call_v(x,D_{t,x})&=\calm_v''(x,D_{t,x})\call_v(x,D_{t,x})=(\id-\calr_v'(x,D_{t,x}))\calm_v'(x,D_{t,x})\call_v(x,D_{t,x})\\
    &=(\id-\calr_v'(x,D_{t,x}))(\id+\calr_v'(x,D_{t,x})) =\id-\calr_v'(x,D_{t,x})^2\\
    &=\id - \calr_v''(x,D_{t,x})=\id - \calr_v(D_{t,x},x),
\end{align*}
which implies the claim.

 \medskip
 
 In the case where $\call_v$ is replaced by $\nonl'(v)p$, an analogous argument applies if we replace $\calm_v'$ by  $\td\calm_{v}':=\frac{1-\phi_0}{\nonl'(v)p}$.
\end{proof}
 \begin{lem}\label{lem_nonl}
 Let $\nonl$ fulfill Assumption \ref{ass_1}\eqref{ass_1ii} for some $m\in(1,\infty)$ and $C>0$.
 Then
 \begin{align}
    \int_{\set{\nonl'(\vf)\le \delta}} |\nonl''(v)| \dd v\le C\delta,  \quad &\tand \quad
    \int_{\set{\nonl'(\vf)> \delta}} \frac{|\nonl''(v)|}{\nonl'(v)^2} \dd v\le C\delta^{-1}, \label{ass_nonl_2}
 \end{align} 
 \end{lem}
 \begin{proof}
  By the co-area formula for $\WSRloc{1}{1}(\R)$ functions (see e.g. Theorem 1.1 in \cite{MSZ03}) we have
    \begin{align*}
     \int_{\set{\nonl'(\vf)>\delta}}\frac{|\nonl''(v)|}{\nonl'(v)^{2}}\dd v
     & =\int\int_{\nonl'^{-1}(y)} \1_{\set{\nonl'(v)>\delta}}\frac{1}{\nonl'(v)^{2}}\dd\mathcal{H}^{0}(v)\dd y\\
     & =\int_{0}^{\infty}\frac{1}{y^{2}} \1_{\set{y>\delta}}\int_{\nonl'^{-1}(y)}\dd\mathcal{H}^{0}(v)\dd y
      =\int_{\delta}^{\infty}\frac{1}{y^{2}}\mathcal{H}^{0}(\nonl'^{-1}(y))\dd y
    \end{align*}
    and
    \begin{align*}
     \int_{\set{\nonl'(\vf)\le\delta}}|\nonl''(v)|\dd v
     & =\int\int_{\nonl'^{-1}(y)} \1_{\set{\nonl'(v)\le\delta}}\dd\mathcal{H}^{0}(v)\dd y\\
     & =\int_{0}^{\infty} \1_{\set{y\le\delta}}\int_{\nonl'^{-1}(y)}\dd\mathcal{H}^{0}(v)\dd y
     =\int_{0}^{\delta}\mathcal{H}^{0}(\nonl'^{-1}(y))\dd y.
    \end{align*}
    Here $\calh^{0}$ is the $0$-dimensional Hausdorff measure on $\R$.
    Observe that $\calh^0(\nonl'^{-1}(\cdot))\in \LRloc{1}(\R)$ by the co-area formula, and that $\mathcal{H}^{0}(\nonl'^{-1}(y))=\#\set{\nonl'(v)=y}$ is estimated by $C$ by assumption.
    Thus \eqref{ass_nonl_2} follows from $\int_0^\delta \dd y = \delta$ and $\int_\delta^\infty \frac{1}{y^2}\dd y = \frac{1}{\delta}$.
 \end{proof}

\begin{lem}\label{lem:av}
 Let the kernel $k$ and the nonlinearity $\nonl$ fulfill Assumption \ref{ass_1} for some $\frex\in(0,2)$, $m\in(1,\infty)$, and let the symbol $p(x,\xi)$ be given by \eqref{fourier-repr-ker}.
Let $f\in \LR{\infty}(\R_t \times \R^d_x \times \R_\vf)$ with $\|f\|_{\LR{\infty}_{t,x,v}}\le 1$
 be a distributional solution to
 \begin{align}\label{av_eqn}
  \diffop_v(x,\partial_t,D_x) f=\g_0+\partial_\vf \g_1  \  \ton \ \R_t \times \R^d_x \times \R_\vf.
 \end{align}
 Here, the pseudodifferential operator $\diffop_v(x,\partial_t,D_x)$ is given in terms of its symbol in \eqref{av_op},
 and $g_0$ and $g_1$ are Radon measures satisfying
 \begin{align*}
|g_0|\in
\calm_{TV}(\R_{t}\times\R_{x}^{d}\times\R_{v}),
\quad
g_1\in \LR{\infty}(\R_{v};\calm_{TV}(\R_{t}\times\R_{x}^{d})).
 \end{align*}
 Let $\overline{f}(t,x):=\int f(t,x,\vf) \dd \vf \in \LR{\infty}_t\LR{1}_x\cap\LR{1}_{t,x}$.
 Then for all $\eta\in\CRi(\R_v)$ such that $\eta \nonl' f\in \LR{1}_{t,x,v}$, $|\eta|\le 1$ and $\eta'\in \LR{1}(\R_v)$ we have $\overline{\eta f}\in S^{(0,\kappa_x)}_{m,\infty}B$, where
 \begin{align}\label{lem:av_constants}
 \begin{split}
  \kappa_x\in (0,\frac{\frex}{m}),
  \end{split}
 \end{align}
 and
 \begin{align}\label{lem:av_est1}
 \begin{split}
  \norm{\overline{\eta f}}_{S^{(0,\kappa_x)}_{m,\infty}B}^m
  &\lesssim \|g_0\|_{\calm_{TV}}+(1+\|\eta'\|_{\LR{1}_v})\|g_1\|_{\LR{\infty}_v\calm_{TV}} \\
  &\quad + \|\overline{f}\|_{\LR{\infty}_t\LR{1}_{x}\cap\LR{1}_{t,x}}^m + \|\eta f\|_{\LR{1}_{t,x,v}} + \|\eta\nonl' f\|_{\LR{1}_{t,x,v}}.
 \end{split}
 \end{align}
 \end{lem}
\begin{proof}
We first assume that $f$, $g_0$, and $g_1$ are compactly supported with respect to the variable $v$.
At the end of the proof, we can remove this additional assumption by a cutoff argument.
We also assume $\eta=1$, an assumption which we can release ourselves from in the same step at the end of the proof.

We decompose $f$ into Littlewood-Paley blocks with respect to both the $t$-variable and the $x$-variable.
Let
$\{\psi_l\}_{l\in\N_0}$ be a partition of unity on $\R\setminus\set{0}$ and
$\{\phi_j\}_{j\in\N_0}$ be a partition of unity on $\R^d\setminus\set{0}$ as in Section \ref{FctSpc}. Then we define for $l,j\in\N_0$
\begin{align*}
 f_{l,j}:=\psi_l(D_t)\phi_j(D_x)f:=\calf^{-1}_{t,x}[\psi_l\phi_j\calf_{t,x}f],
\end{align*}
so that $\calf_{t,x}{f}_{l,j}(\tau,\xi,v)$ is supported on frequencies $|\tau|\sim2^l$ and $|\xi|\sim2^{j}$ for $l,j\in\N$.
The low spatial frequency part $f_{l,j}$, $j\le 2$ can be estimated in view of
 Bernstein's Lemma (cf.\@ \cite[Lemma 2.1]{BCD11}) and
 $\|\phi_j\|_{\calm^1},\|\psi_l\|_{\calm^1}\lesssim 1$ via
 \begin{align}\label{lem:av_est_Bernstein}
  \|\overline{f}_{l,j}\|_{\LR{m}_{t,x}} \lesssim \|\overline{f}\|_{\LR{m}_{t}\LR{1}_x}\lesssim \|\overline{f}\|_{\LR{\infty}_t\LR{1}_{x}\cap\LR{1}_{t,x}}.
 \end{align}

\newcounter{avl} 
\refstepcounter{avl}\label{avlS1} 
\smallskip\noindent
\textit{Step \arabic{avl}\refstepcounter{avl}\label{avlS2}:} Degenerate part. 
Let $l\in\N_0$ and $j\ge 2$.
We note that we have the estimate $\|f_{l,j}\|_{\LR{\infty}_{t,x,v}}\lesssim \|f\|_{\LR{\infty}_{t,x,v}}\le 1$ with a constant independent of $l$ and $j$, since $\|\vp_j\|_{\calm^\infty}=\|\vp_1\|_{\calm^\infty}<\infty$ and similarly for $\psi_l$. Thus, we have for any $\d>0$
\begin{align*}
 \left\|\int_{\set{\nonl'(v)\le \d}} f_{l,j}\dd v\right\|_{L_{t,x}^{\infty}}& \le \int_{\set{\nonl'(v)\le \d}} \|f_{l,j}(\cdot,v)\|_{\LR{\infty}_{t,x}}\dd v \lesssim |\set{\nonl'(v)\le \d}|\lesssim \d^{\frac{1}{m-1}}.
\end{align*}

\smallskip\noindent
\textit{Step \arabic{avl}\refstepcounter{avl}\label{avlS3}:} Nondegenerate part.
Define for all $v\in \R$ with $\nonl'(v)\ne 0$ the symbols $\calm_v$ and $\calr_v$ as in Lemma \ref{lem:FM}. 
Then there holds the representation
\begin{align*}
f &=  \calm_{v}(D_{t,x},x) [g_0 +\partial_v g_1] + \calr_{v}(D_{t,x},x)f.
\end{align*}
For $j,l\ge 2$, the operators $\calm_{vlj}(D_{t,x},x)$ and $\partial_v\calm_{vlj}(D_{t,x},x)$ extend by Lemma \ref{lem:FM} and Theorem \ref{thm:FM} to bounded operators on $\calm_{TV}$ of order $2^{-\frex j}/\nonl'(v)$ and $2^{-\frex j}\frac{|\nonl''(\vf)|}{\nonl'(\vf)^{2}}$,
and we recall that by Lemma \ref{lem_nonl} we have
\begin{align}\label{eq:av1}
    \int_{\set{\nonl'(v)\ge \d}} \frac{|\nonl''(v)|}{\nonl'(\vf)^{2}} \dd v
    \lesssim \d^{-1}.
\end{align}

Let $\psi\in \CRci(\R_t\times \R^d_x)$ with $\|\psi\|_{\LR{\infty}_{t,x}}\le 1$ and $\d>0$.
For $\eps>0$ we consider
\begin{align*}
    \eta^{\eps}(v)=(\vp^{\eps}\ast \1_{\set{\nonl'(v)\ge\d}})(v),
\end{align*}
where $\vp^{\eps}$ is a smooth Dirac sequence.
We then have
\begin{align*}\int_{\R_{v}}\eta^{\eps}\int_{\R_{t}\times\R_{x}^{d}}\psi f_{l,j}\dd t\dd x\dd v & =\int_{\R_{v}}\eta^{\eps}\int_{\R_{t}\times\R_{x}^{d}}\psi\big(\calm_{vlj}(D_{t,x},x)g_0-\partial_{v}\calm_{vlj}(D_{t,x},x)g_1\big)\dd t\dd x\dd v\\
 & \quad+\int_{\R_{v}}\partial_{v}\eta^{\eps}\int_{\R_{t}\times\R_{x}^{d}}\psi\calm_{vlj}(D_{t,x},x)g_1\dd t\dd x\\
 & \quad+\int_{\R_{v}}\eta^{\eps}\int_{\R_{t}\times\R_{x}^{d}}\psi\calr_{vlj}(D_{t,x},x)f\dd v.
\end{align*}
We can pass to the limit $\eps\to 0$ in all terms on the right hand side except for the one containing $\partial_{v}\eta^{\eps}$.
For this term we use the co-area formula in \cite[Theorem 1.1]{MCS01}
\begin{align*}
\partial_{v}\eta^{\eps}(v) & =\int_{\R_{w}}\vp^{\eps}(v-w)\partial_{w} \1_{\set{\nonl'(v)\ge\d}}(w)\dd w =\int_{\set{\nonl'(v)=\d}}\vp^{\eps}(v-w)\dd\mathcal{H}^{0}(w),
\end{align*}
where $\mathcal{H}^0$ is the $0$-dimensional Hausdorff measure on $\R$.
We hence may write
\begin{align*}
\int_{\R_{v}}\partial_{v}\eta^{\eps}&\int_{\R_{t}\times\R_{x}^{d}}\psi\calm_{vlj}(D_{t,x},x)g_1\dd t\dd x \\
&=\int_{\R_{v}}\int_{\set{\nonl'(v)=\d}}\vp^{\eps}(v-w)\dd\mathcal{H}^{0}(w)\int_{\R_{t}\times\R_{x}^{d}}\psi\calm_{vlj}(D_{t,x},x)g_1\dd t\dd x\\
 & =\int_{\R_{v}}\int_{\set{\nonl'(v)=\d}}\vp^{\eps}(v-w)\int_{\R_{t}\times\R_{x}^{d}}\psi\calm_{vlj}(D_{t,x},x)g_1\dd t\dd x\dd\mathcal{H}^{0}(w),
\end{align*}
and thus estimate with $\|\psi\|_{\LR{\infty}_{t,x}}\le 1$ and $\|\calm_{vlj}\|_{\calm_{TV}\to\calm_{TV}}\lesssim \frac{2^{-\frex j}}{\nonl'(v)}$
\begin{align*}
\left|\int_{\R_{v}}\partial_{v}\eta^{\eps}\int_{\R_{t}\times\R_{x}^{d}}\psi\calm_{vlj}(D_{t,x},x)g_1\dd t\dd x\right|
    & \lesssim\int_{\R_{v}}\int_{\set{\nonl'(v)=\d}}\vp^{\eps}(v-w)\frac{2^{-\frex j}}{\nonl'(v)}\|g_1(v)\|_{\mathcal{M}_{TV}}\dd\mathcal{H}^{0}(w)\\
 & \le 2^{-\frex j}\|g_1\|_{L_{v}^{\infty}\mathcal{M}_{TV}}\int_{\set{\nonl'(v)=\d}}(\vp^{\eps}\ast\frac{1}{\nonl'(v)})(w)\dd\mathcal{H}^{0}(w).
\end{align*}
Therefore we deduce
\begin{align*}
\limsup_{\eps\to0}\int_{\R_{v}}\partial_{v}\eta^{\eps}&\int_{\R_{t}\times\R_{x}^{d}}\psi\calm_{vlj}(D_{t,x},x)g_1\dd t\dd x \\
& \lesssim 2^{-\frex j}\|g_1\|_{L_{v}^{\infty}\mathcal{M}_{TV}}\int_{\set{\nonl'(v)=\d}}\limsup_{\eps\to0}(\vp^{\eps}\ast\frac{1}{|\nonl'(v)|})(w)\dd\mathcal{H}^{0}(w)\\
 & \le 2^{-\frex j}\|g_1\|_{L_{v}^{\infty}\mathcal{M}_{TV}}\int_{\set{\nonl'(v)=\d}}\frac{1}{|\nonl'(w)|}\dd\mathcal{H}^{0}(w)\\
 & \le 2^{-\frex j}\delta^{-1}\|g_1\|_{L_{v}^{\infty}\mathcal{M}_{TV}}\mathcal{H}^{0}(\set{\nonl'(v)=\d})
 \lesssim 2^{-\frex j}\delta^{-1}\|g_1\|_{L_{v}^{\infty}\mathcal{M}_{TV}}.
\end{align*}

Since the operator $\calr_{vlj}(x,D_{t,x})$ extends by Lemma \ref{lem:FM} and Theorem \ref{thm:FM} to a constant multiplier on $\LR{1}_{t,x}$ of order $2^{-\frex j}$, we conclude
\begin{align*}
 &\|\int_{\set{\nonl'(v)\ge \d}} f_{l,j} \dd v\|_{L_{t,x}^{1}} \\
 &\quad \le \int_{\set{\nonl'(v)\ge \d}} \left(\|\calm_{vlj}(D_{t,x},x)g_0\|_{\LR{1}_{t,x}} + \|\partial_v\calm_{vlj}(x,D_{t,x})g_1\|_{\LR{1}_{t,x}}\right) \dd v \\
 &\qquad + 2^{-\frex j}\delta^{-1}\|g_1\|_{\LR{\infty}_v\calm_{TV}} +  \delta^{-1}\|\calr_{vlj} \nonl' f\|_{\LR{1}_{t,x,v}} \\
 &\quad \lesssim  2^{-\frex j}\d^{-1}\left(\|g_0\|_{\calm_{TV}} + \d\int_{\set{\nonl'(v)\ge \d}} \frac{|\nonl''(v)|}{\nonl'(v)^2} \|g_1(\cdot,v)\|_{\calm_{t,x}} \dd v + \|g_1\|_{\LR{\infty}_v\calm_{TV}}  + \|\nonl' f\|_{\LR{1}_{t,x,v}}\right) \\
 &\quad \lesssim 2^{-\frex j}\d^{-1}(\|g_0\|_{\calm_{TV}} + \|g_1\|_{\LR{\infty}_v\calm_{TV}} + \|\nonl' f\|_{\LR{1}_{t,x,v}}).
\end{align*}
In the case $l\le 1$ and $j\ge 2$, it will be more advantageous to leverage upon $\td\calm_v$ and $\td\calr_v$ from Lemma \ref{lem:FM} to obtain the slightly different decomposition
\begin{align}\label{eq:eqn_fK-1-2}
f =  \td\calm_{v}(D_{x},x) [g_0 +\partial_v g_1 - \calf_t^{-1}i\tau \calf_t f] + \td\calr_{v}(D_{x},x)f.
\end{align}
We use that $\psi_l(D_t)$ commutes with all pseudo-differential operators considered here (since their dependence is only in space), so that in view of $\psi_l=\psi_l(\psi_0+\psi_1+\psi_2)$ we have
\begin{align*}
   \|\td\calm_{vlj}(D_x,x)\calf_t^{-1}i\tau\calf_tf\|_{\LR{1}_{t,x}}
   &\lesssim 2^{-\frex j} \nonl'(v)^{-1}\|\calf_t^{-1}i\tau(\psi_0+\psi_1+\psi_2)\calf_tf\|_{\LR{1}_{t,x}}\\
   &\lesssim 2^{-\frex j} \nonl'(v)^{-1}\|f\|_{\LR{1}_{t,x}}.
\end{align*}
Hence, by the same arguments as above, we arrive also in the case $l\le 1$ at
\begin{align*}
 &\|\int_{\set{\nonl'(v)\ge \d}} f_{l,j} \dd v\|_{L_{t,x}^{1}} \\
 &\le \int_{\set{\nonl'(v)\ge \d}} \Big(\|\td\calm_{vlj}(D_x,x)g_0\|_{\LR{1}_{t,x}} + \|\partial_v\td\calm_{vlj}(D_x,x)g_1\|_{\LR{1}_{t,x}} + \|\td\calm_{vlj}(D_x,x)\calf_t^{-1}i\tau\calf_tf\|_{\LR{1}_{t,x}}\Big) \dd v  \\
 &\qquad + 2^{-\frex j}\delta^{-1}\|g_1\|_{\LR{\infty}_v\calm_{TV}} + \delta^{-1} \|\td\calr_v(D_x,x)\nonl' f\|_{\LR{1}_{t,x,v}} \\
 & \lesssim  2^{-\frex j}\d^{-1}\left(\|g_0\|_{\calm_{TV}} + \d\int_{\set{\nonl'(v)\ge \d}} \frac{|\nonl''(v)|}{\nonl'(v)^2} \|g_1(\cdot,v)\|_{\calm_{t,x}} \dd v +\|f\|_{\LR{1}_{t,x}} + \|g_1\|_{\LR{\infty}_v\calm_{TV}}  + \|\nonl' f\|_{\LR{1}_{t,x,v}}\right) \\
 & \lesssim 2^{-\frex j}\d^{-1}(\|g_0\|_{\calm_{TV}} + \|g_1\|_{\LR{\infty}_v\calm_{TV}} + \|f\|_{\LR{1}_{t,x,v}} +\|\nonl' f\|_{\LR{1}_{t,x,v}}).
\end{align*}

\smallskip\noindent
\textit{Step \arabic{avl}\refstepcounter{avl}\label{avlS4}:} Conclusion. 
Let $\kappa_x\in (0,\frac{\frex}{m})$. We aim to conclude by real interpolation. We set, for $\tauz>0$,
\begin{align*}
K(\tauz,\overline{f}_{l,j}):=\inf\{ & \|\overline{f}_{l,j}^{1}\|_{\LR{1}_{t,x}}+\tauz\|\overline{f}_{l,j}^{0}\|_{\LR{\infty}_{t,x}}:\overline{f}_{l,j}^{0}\in \LR{\infty}_{t,x}, \overline{f}_{l,j}^{1}\in \LR{1}_{t,x},\ \overline{f}_{l,j}=\overline{f}_{l,j}^{0}+\overline{f}_{l,j}^{1}\}.
\end{align*}
By the above estimates we obtain with $c:=\|g_0\|_{\calm_{TV}}+\|g_1\|_{\LR{\infty}_v\calm_{TV}} +\|f\|_{\LR{1}_{t,x,v}}+ \|\nonl' f\|_{\LR{1}_{t,x,v}}$ (by relabelling $\d\mapsto \d^{m-1}$)
\begin{align*}
K(\tauz,\overline{f}_{l,j}) & \lesssim 2^{-\frex j}\d^{1-m}c +\tauz \d.
\end{align*}
We now equilibrate the first and the second term on the right hand side, that is, we choose $\delta>0$ such that
\begin{align*}
 2^{-\frex j}\d^{1-m}c=\tauz\d,
 \quad \text{that is} \quad
 \d
 =2^{-\frac{\frex j}{m}}c^{\frac{1}{m}} \tauz^{-\frac{1}{m}}.
\end{align*}
Hence, with 
\[
\t:=1-\frac{1}{m}
\]
we obtain
\begin{align*}
\tauz^{-\t}K(\tauz,\overline{f}_{j})& \lesssim  2^{-\frac{\frex j}{m}}c^{\frac{1}{m}}.
\end{align*}
Observe that $1-\t + \frac{\t}{\infty}=1-\t$, so that $(\LR{1}_{t,x},\LR{\infty}_{t,x})_{\t,\infty}=\LR{m,\infty}_{t,x}$, since $m=\frac{1}{1-\t}$. Hence, we may take the supremum over $\tauz>0$ to obtain
\begin{align}\label{av_est_lj}
\|\overline{f}_{l,j}\|_{\LR{m,\infty}_{t,x}}  \lesssim 2^{-\frac{\frex j}{m}}c^{\frac{1}{m}}.
\end{align}
Let $\sigma_x\in (\kappa_x,\frac{\frex}{m})$. Choose $\vartheta\in(0,1)$ such that $\frac{\frex \vartheta}{m}> \sigma_x$ and $\frac{1}{q}:=1-\vartheta + \frac{\vartheta}{m}$ is such that
\begin{align*}
 d(\frac{1}{q} - \frac{1}{m})=d (1-\vartheta)(1-\frac{1}{m})< \frac{\frex \vartheta}{m} - \sigma_x.
\end{align*}
Since $q<m$ and Lorentz spaces are increasing in their second parameter, we have $\LR{m,q}_t\subset \LR{m}_t$. Using additionally the interpolation results $(\LR{1}_{t,x},\LR{m,\infty}_{t,x})_{\vartheta,q}=\LR{q}_{t,x}$, $(\LR{\infty}_{t}\LR{1}_{x},\LR{m,\infty}_{t,x})_{\vartheta,q}=\LR{\frac{m}{\vartheta},q}_t\LR{q}_x$ as well as $(\LR{\frac{m}{\vartheta},q}_{t}\LR{q}_x,\LR{q}_{t,x})_{\frac{1}{m},q}=\LR{m,q}_t\LR{q}_x$, together with \eqref{av_est_lj}, we obtain in view of Bernstein's Lemma and $j\ge 0$
\begin{align}\label{av_est_lj_2}
\begin{split}
\|\overline{f}_{l,j}\|_{\LR{m}_{t,x}}
&\lesssim 2^{d(\frac{1}{q}-\frac{1}{m})j}\|\overline{f}_{l,j}\|_{\LR{m}_t\LR{q}_{x}}
\lesssim 2^{d(\frac{1}{q}-\frac{1}{m})j}\|\overline{f}_{l,j}\|_{\LR{m,q}_t\LR{q}_{x}} \\
&\lesssim 2^{d(\frac{1}{q}-\frac{1}{m})j}(\|\overline{f}_{l,j}\|_{\LR{\frac{m}{\vartheta},q}_t\LR{q}_{x}} + \|\overline{f}_{j}\|_{\LR{q}_{t,x}}) \\
&\lesssim 2^{(\frac{\frex \vartheta}{m} - \sigma_x)j}\|\overline{f}_{l,j}\|_{\LR{\infty}_{t}\LR{1}_{x}\cap \LR{1}_{t,x}}^{1-\vartheta}\|\overline{f}_{l,j}\|_{\LR{m,\infty}_{t,x}}^\vartheta \\
& \lesssim \|\overline{f}\|_{\LR{\infty}_{t}\LR{1}_{x}\cap \LR{1}_{t,x}}^{1-\vartheta}2^{- \sigma_x j}c^{\frac{\vartheta}{m}}  \\
&\le 2^{-\sigma_x j}(c^\frac{1}{m}  + \|\overline{f}\|_{\LR{\infty}_{t}\LR{1}_{x}\cap \LR{1}_{t,x}}).
\end{split}
\end{align}
Multiplying by $2^{\sigma_x j}$ and taking the supremum over $j\in \N_0$ yields (in view of \eqref{lem:av_est_Bernstein} for $j\in\set{0,1}$)
\begin{align*}
 &\norm{\overline{f}}_{S^{(0,\sigma_x)}_{m,\infty}B}\lesssim c^{\frac{1}{m}} +  \|\overline{f}\|_{\LR{\infty}_{t}\LR{1}_x\cap \LR{1}_{t,x}}.
\end{align*}
This is \eqref{lem:av_est1}.

\textit{Step \arabic{avl}\refstepcounter{avl}:} 
It remains to consider the case when $f$, $g_0$, and $g_1$ are not localized in $v$ and $\eta\ne 1$. We observe that for a smooth cut-off function $\psi\in \CRci(\R)$, the function $(t,x,v)\to f(t,x,v)\psi(v)=:f^\psi(t,x,v)$ is a solution to
\begin{align*}
  \op[\diffop_\vf] f^\psi(t,x,\vf)=\g_0^\psi(t,x,\vf)+g_1^{\psi'}(t,x,v)+\partial_\vf g_1^\psi(t,x,\vf) \ \ton \R_t \times \R^d_x \times \R_\vf,
 \end{align*}
 where $\g^\psi$, $q^{\psi'}$ and $\q^\psi$ are defined analogously.
 Hence, estimate \eqref{av_est_lj_2} reads in this case
 \begin{align*}
\|\overline{f^\psi_{lj}}\|_{\LR{m}_{t,x}} \lesssim 2^{-\sigma_x j}((\|(g_0^\psi+g_1^{\psi'})\|_{\calm_{TV}}+\|g_1^\psi\|_{\LR{\infty}_v\calm_{t,x}}+\|f^\psi\|_{\LR{1}_{t,x,v}}+\|\nonl' f^\psi\|_{\LR{1}_{t,x,v}})^{\frac{1}{m}} + \|\overline{f}\|_{\LR{\infty}_{t}\LR{1}_{x}\cap \LR{1}_{t,x}}).\end{align*}
  For $n\in\N$ and a smooth cut-off function $\psi\in \CRci(\R)$ with $\psi=1$ on $B_1(0)$, $\supp\psi\subset B_2(0)$ and $0\le\psi\le 1$, we define $\psi_n$ via $\psi_n(v):=\psi(v/r_n)$. Hence $\psi'_n$ is supported on $r_n\le |v|\le 2r_n$ and takes values in $[0,1/r_n]$, so that we may estimate for $\eta\in \CRi(\R_v)$ with $|\eta|\le 1$
 \begin{align*}
  \|g_1^{(\eta\psi_n)'}\|_{\calm_{TV}}&\le \int_{\R_t\times \R^d_x\times \R_v} |\psi'_n(v)\eta(v)| |g_1|\dd v \dd x \dd t + \int_{\R_t\times\R^d_x\times \R_v} |\psi_n(v)\eta'(v)| |g_1|\dd v \dd x \dd t \\
  &\le \int_{\R_t\times \R^d_x\times \R_v} \chi_{\set{r_n\le |v|\le 2 r_n}}|\psi'_n(v)| |g_1|\dd v \dd x \dd t + \|g_1\|_{\LR{\infty}_v\calm_{t,x}}\|\eta'\|_{\LR{1}_v}\\
  &\lesssim \|g_1\|_{\LR{\infty}_v\calm_{t,x}}\int_{\R_v} \chi_{\set{r_n\le |v|\le 2 r_n}} |\psi'_n(v)| \dd v + \|g_1\|_{\LR{\infty}_v\calm_{t,x}}\|\eta'\|_{\LR{1}_v}\\
  &\le (1+\|\eta'\|_{\LR{1}_v})\|g_1\|_{\LR{\infty}_v\calm_{t,x}}.
 \end{align*}
 Thus, taking the limit $n\to\infty$ and using Fatou's lemma, we obtain \eqref{av_est_lj_2} also for general $f$, $g_0$, and $g_1$ and we may conclude as before.
\end{proof}

\begin{lem}\label{lem:av3}
 Assume $m\in \vpp{1,\infty}$
 and let $\nonl$, $f$, $g_0$, $g_1$, and $\overline{f}$ be as in Lemma \ref{lem:av}.
 Then for all $\eta\in\CRi(\R_v)$ such that $\eta \nonl' f\in \LR{1}_{t,x,v}$, $|\eta|\le 1$ and $\eta'\in \LR{1}(\R_v)$ we have $\overline{\eta f}\in S^{(1,0)}_{1,\infty}B$
 and
 \begin{align}\label{lem:av3_est}
  &\norm{\overline{\eta f}}_{S^{(1,0)}_{1,\infty}B}\lesssim \|g_0\|_{\calm_{TV}} +(1+\|\eta'\|_{\LR{1}_v})\|g_1\|_{\LR{\infty}_v\calm_{TV}} + \|f\|_{\LR{1}_{t,x,v}} + \|\eta \nonl' \f\|_{\LR{1}_{t,x,v}}.
 \end{align}
 \end{lem}
\begin{proof}
 We first assume that $f$ is compactly supported in $v$ and $\eta=1$.
 For $l,j\in \N_0$, we introduce $f_{l,j}:=\calf^{-1}_{t,x}\psi_l(\tau)\phi_j(\xi)\calf_{t,x} f$.
 The low time-frequency part $f_{l,j}$ with $l\le 1$ can be estimated in view of
 $\|\psi_l\|_{\calm^1}\lesssim 1$ and $\|\phi_j\|_{\calm^1}=\|\phi_1\|_{\calm^1}\lesssim 1$ via
 \begin{align}\label{lem:av3_est5}
  \|\overline{f}_{l,j}\|_{\LR{1}_{t,x}} \lesssim \|\overline{f}\|_{\LR{1}_{t,x}} = \|f\|_{\LR{1}_{t,x,v}}.
 \end{align}
 Next, we estimate $f_{l,j}$ for $l,j\ge 2$.
 We leverage upon the fact that $f$ solves \eqref{av_eqn} to write
 \begin{align*}
  \int f_{l,j}\dd v
 &  =\int \calm_{vlj}(D_{t,x},x)g_0\dd v 
 -\int  \partial_v\calm_{vlj}(D_{t,x},x) g_1 \dd v + \int \calr_{vlj}(D_{t,x},x)f\dd v,
 \end{align*}
 where $\calm_v(x,\tau,\xi)$ and $\calr_v(x,\tau,\xi)$ are defined as in \eqref{lem:FMe001}.
By Lemma \ref{lem:FM} and Theorem \ref{thm:FM} the operator $\calr_{vlj}(D_{t,x},x)$ extends to a bounded operator on $\LR{1}_{t,x}$ with a bound estimated by $2^{-l}(1+\nonl'(v))$, while $\calm_{vlj}$ extends to a bounded operator on $\calm_{TV}$ with a bound estimated by $2^{-l}$, and $\partial_v\calm_{vlj}$ extends to a bounded operator on $\calm_{TV}$ with a bound estimated by $\min\set{2^{-\frex j}\frac{|\nonl''(\vf)|}{\nonl'(\vf)^2},2^{-2l}2^{\frex  j}|\nonl''(\vf)|}$.

Since we have
by Lemma \ref{lem_nonl}
\begin{align*}
 \int_{\set{\nonl'(v)\ge \d}} \frac{|\nonl''(\vf)|}{\nonl'(\vf)^2}\dd v\lesssim\d^{-1},
 \quad
 \int_{\set{\nonl'(v)< \d}} |\nonl''(\vf)|\dd v\lesssim \d,
\end{align*}
we obtain for any $\d>0$
\begin{align*}
 \norm{\int f_{l,j} \dd v}_{\LR{1}_{t,x}} & \lesssim 2^{-l}\|g_0\|_{\calm_{TV}} + (2^{-\frex j}\d^{-1} + 2^{-2l}2^{\frex j}\d)\|g_1\|_{\LR{\infty}_v\calm_{TV}} + 2^{-l} (\|f\|_{\LR{1}_{t,x,v}} + \|\nonl' f\|_{\LR{1}_{t,x,v}}).
\end{align*}
We now equilibrate the right-hand side by choosing $\d:= 2^{l}2^{-\frex j}$, so that
\begin{align}\label{lem:av3_est4}
 \norm{\int f_{l,j} \dd v}_{\LR{1}_{t,x}} & \lesssim 2^{-l}(\|g_0\|_{\calm_{TV}} +\|g_1\|_{\LR{\infty}_v\calm_{TV}} + \|f\|_{\LR{1}_{t,x,v}} + \|\nonl' f\|_{\LR{1}_{t,x,v}}).
\end{align}
 It remains to estimate the contribution of $f_{l,j}$ for $j\le 1$. Observe that $f_{l,j}$ solves the equation
  \begin{align*}
   f_{l,j}= & -\nonl'(\vf)\phi_j(D_x)p_R(D_x,x)\calf^{-1}_{t}\frac{\psi_l(\tau)}{i\tau}\calf_{t}\f + \calf^{-1}_{t}\frac{1}{i\tau}\calf_{t}  (g_0)_{l,j} + \calf^{-1}_{t}\frac{1}{i\tau}\calf_{t} \partial_\vf (g_1)_{l,j}.
  \end{align*}
  Integrating in $v$, we obtain
  \begin{align*}
   \overline{f}_{l,j}=- \int \nonl'(\vf)\phi_j(D_x)p_R(D_x,x)\calf^{-1}_{t}\frac{\psi_l(\tau)}{i\tau}\calf_{t}\f \dd v + \calf^{-1}_{t}\frac{1}{i\tau}\phi_0(\xi)\calf_{t} \int (g_0)_{l,j} \dd v.
  \end{align*}
  Since $p(x,\xi)$ acts as a constant multiplier of order $1$ on the support of $\phi_0$ and $\phi_1$, and since $\tau^{-1}$ acts as a constant multiplier of order $2^{-l}$ on the support of $\psi_l$, it follows for $j\le 1$
  \begin{align}\label{lem:av3_est6}
   \norm{\overline{f}_{l,j}}_{\LR{1}_{t,x}}\lesssim 2^{-l}(\|\nonl' \f\|_{\LR{1}_{t,x,v}} + \|g_0\|_{\calm_{TV}}).
  \end{align}
 Collecting \eqref{lem:av3_est5}, \eqref{lem:av3_est4} and \eqref{lem:av3_est6}, multiplying by $2^{l}$ and taking the supremum over $l\in\Z$, we arrive at \eqref{lem:av3_est} for $\eta=1$ and for such $f$ which are localized in $v$. 
 Applying the arguments from Step \ref{avlS4} in the proof of Lemma \ref{lem:av}, we obtain the full assertion.
\end{proof}

%

\section{\label{app:kin_solutions}Generalized kinetic solutions}

In \cite{AAO20}, the kinetic theory for homogeneous, nonlocal parabolic-hyperbolic PDEs was studied. Here, we generalize this theory to spatially inhomogeneous kernels $k$, that is to kernels which depend on both $x$ and $y$ instead of just $y$. 

In the following, we will first introduce the notion of entropy and kinetic solutions to PDEs of the type
\begin{align}\label{eq:par-hyp}
\begin{pdeq}
\partial_{t}u+\div F(u)+L \nonl(u) & =S &&\text{on }(0,T)\times\R_{x}^{d}\\
u(0) & =u_{0} && \text{on }\R_{x}^{d},
\end{pdeq}
\end{align}
where $L$ is of the pure jump type \eqref{def_op} and
\begin{equation}\begin{split}\label{eqn:sec5_as}
&u_{0}  \in L^{1}(\R_{x}^{d}),\,S\in L^{1}([0,T]\times\R_{x}^{d}),\,T\ge0, \\
&\nonl  \in \WSRloc{1}{\infty}(\R) \text{ is nondecreasing}, \\
&F  \in \WSRloc{1}{\infty}(\R,\R^d),  \\
&k:\R^d\times(\R^d\setminus\set{0})\to(0,\infty) \text{ measurable with } k(x,y)=k(x+y,-y) \text{ for all } x,y\in\R^d, \\
&\sup_{x\in\R^d}\int_{\R^d} ((\1_{[1,2)}(a)|y|^2 + \1_{(0,1)}(a)|y|) \wedge 1)k(x,y)\dd y<\infty. 
\end{split}\end{equation}
Motivated from this, we introduce a broad class of so-called generalized kinetic solutions. This constitutes an extension to parabolic-hyperbolic, nonlocal PDEs of related concepts for  scalar conservation laws by \cite[Chapter 4]{Perthame2002}. Despite the breadth of the resulting class of solutions, we demonstrate in this section that it offers sufficient structure to be amenable to the regularity results derived in this work. 

We start by recalling the motivation of the concept of entropy solutions. Let $s\in C^2(\R)$ be convex and $\eta:\R\to\R^d$, $\beta:\R\to\R$ such that $\eta'=s'F'$ and $\beta'=s'\nonl'$. Moreover let $\varphi$ be smooth, compactly supported on $[0,T)\times \R^d$ and nonnegative. Then, we have for any smooth solution $u$ to \eqref{eq:par-hyp} that

\begin{align*}
 -\int_0^T\int_{\R^d} &s'(u) S \dd x\dd t=-\int_0^T\int_{\R^d} s'(u)\varphi \big(\partial_t u +\div F(u)+L \nonl(u)\big) \dd x \dd t \\
&= \int_0^T\int_{\R^d}  \big(s(u)\partial_t \varphi + \eta(u)\cdot\nabla\varphi) \dd x \dd t + \int_{\R^d} s(u_0)\varphi(0,x) \dd x \\
&\quad + \int_0^T\int_{\R^d} \int_{\R^d} s'(u)\varphi  \big(\tau_y\nonl(u)-\nonl(u)-\1_{[1,2)}(\frex)\1_{|y|\le 2}(y) y\cdot \nabla\nonl(\uf)\big) k(x,y) \dd y \dd x \dd t.
\end{align*}
We next use
\begin{align*}
s'(a)(\nonl(b)-\nonl(a))&=\beta(b)-\beta(a) - \int_a^b s''(v)|\nonl(b)-\nonl(v)| \dd v\\
& = \beta(b)-\beta(a) - \int_\R s''(v)|\nonl(b)-\nonl(v)|\1_{\mathrm{conv}\set{a,b}}(v) \dd v,
\end{align*}
which can be derived by Taylor's formula and the fact that $\nonl$ is non-decreasing in order to introduce the absolute value, see \cite[Lemma 22]{AAO20}. Here $\mathrm{conv}\set{a,b}=(\min\set{a,b},\max\set{a,b})$. 
Hence
\begin{align*}
    \int_0^T\int_{\R^d} \int_{\R^d} &s'(u)\varphi(t,x) \big(\tau_y\nonl(u)-\nonl(u)-\1_{[1,2)}(\frex)\1_{|y|\le 2}(y) y\cdot \nabla\nonl(\uf)\big) k(x,y) \dd y \dd x \dd t \\
    &= \int_0^T\int_{\R^d} \beta(u) L\varphi \dd x \dd t - \int_0^T\int_{\R^d} \int_\R s''(v)n(t,x,v)\varphi \dd v \dd x \dd t,
\end{align*}
where we have used that $L$ is $L^2$-symmetric due to $k(x,y)=k(x+y,-y)$ in the last step, and where
\begin{equation}\begin{split}\label{nonl-dis-eq}
n(t,x,v)
&:=\int_{\R^d} |\nonl(u(t,x+y))-\nonl(v)| \1_{\mathrm{conv}\set{\tau_y u(t,x),u(t,x)}}(v) k(x,y) \dd y 
\end{split}\end{equation}
is the nonlocal dissipation measure.
It follows
\begin{align}
\begin{split}\label{entropy-sol-eq}
&\int_0^T\int_{\R^d}  \big(s(u)\partial_t \varphi + \eta(u)\cdot\nabla\varphi - \beta(u)L\varphi\big)\dd x \dd t + \int_{\R^d} s(u_0)\varphi(0,x) \dd x \\
& = \int_0^T\int_{\R^d} \int_\R s''(v)n(t,x,v)\varphi \dd v \dd x \dd t - \int_0^T\int_{\R^d} s'(u) S \dd x \dd t.
\end{split}
\end{align}

This discussion motivates the following definition of entropy solutions.

\begin{defn}\label{def:entropy-sol}
 Let $T\ge 0$, $u_0\in L^1(\R^d_x)\cap L^\infty (\R^d_x)$ and $S\in L^1([0,T]\times \R^d_{x})$.
 We call a function $u\in C([0,T];L^1(\R^d_x))\cap L^\infty([0,T]\times \R^d_x)$ entropy solution to \eqref{eq:par-hyp} if for all $(s,\eta,\beta)$ with $s\in C^2(\R)$, $\eta:\R\to\R^d$ and $\beta:\R\to\R$ with $\eta'=s'F'$ and $\beta'=s'\nonl'$ and all nonnegative, smooth $\varphi$ with compact support in  $[0,T)\times \R^d_x$
  there holds
\begin{align}
\begin{split}\label{entropy-sol-ineq}
&\int_0^T\int_{\R^d}  \big(s(u)\partial_t \varphi + \eta(u)\cdot\nabla\varphi - \beta(u)L\varphi) \dd x \dd t + \int_{\R^d} s(u_0)\varphi(0,x) \dd x \\
& \ge \int_0^T\int_{\R^d} \int_\R s''(v)n(t,x,v)\varphi \dd v \dd x \dd t - \int_0^T\int_{\R^d} s'(u) S \dd x \dd t,
\end{split}
\end{align} 
with the nonlocal dissipation measure $n$ defined in \eqref{nonl-dis-eq}.
\end{defn}

Formally choosing $s(u):=|u-v|$ for $v\in\R$ and then differentiating with respect to $v$ yields the concept of kinetic equation in $\chi(u(t,x),v):=\frac12\partial_v(|v|-|u(t,x)-v|)$.

\begin{defn}
\label{def:kinetic_sol-1} Let $T\ge 0$, $u_0\in L^1(\R^d_x)$ and $S\in L^1([0,T]\times \R^d_{x})$. We say that $u\in C([0,T];L^{1}(\R^{d}_x))$ is a kinetic solution to \eqref{eq:par-hyp} if the corresponding kinetic function $\chi$ satisfies the following:
\begin{enumerate}
\item There is a non-negative measure $m\in\calm^{+}_{t,x,v}$ such that, in the sense of distributions,
\[
\partial_{t}\chi+F'(v)\cdot\nabla_{x}\chi-\nonl'(v)L\chi=\partial_{v}q+\delta_{v=u(t,x)}S\quad\text{on }(0,T)\times\R_{x}^{d}\times\R_{v}
\]
where $q=m+n$, and $n$ is the nonlocal dissipation measure defined in \eqref{nonl-dis-eq}.
\item We have
\begin{equation}\label{eq:decay_dissipation_measures_kin}
\int q(t,x,v)\dd x \dd t\le\mu(v)\in L_{0}^{\infty}(\R),
\end{equation}
where $L_{0}^{\infty}$ is the space of $L^{\infty}$-functions vanishing for $|v|\to\infty$.
\end{enumerate}
\end{defn}

\begin{rem}  A function $u\in C([0,T];L^{1}(\R^{d}_x)) \cap L^\infty([0,T]\times\R^d_x)$ is a kinetic solution if and only if it is an entropy solution.
\end{rem}
The proof of this remark follows along the same line as \cite[Theorem 12]{AAO20}.

\medskip

Since, in the present work, we are deriving regularity results for solutions to \eqref{eq:par-hyp} with $F\equiv 0$, we next introduce a comparatively large class of solutions, requiring only what is needed in the proof of regularity. This notion of generalized kinetic solutions is inspired by \cite[Chapter 4]{Perthame2002}. 

A decisive difference to kinetic solutions is that for generalized kinetic solutions it is not required anymore that the solution takes the form of a kinetic function $f(t,x,v)=\chi(u(t,x),v)$. The reason is that the latter expression is unstable with respect to weak convergence of sequences $u^n$. As a result, the class of generalized kinetic solutions offers favorable stability properties with respect to weak convergence compared to kinetic solutions. 

In contrast to \cite[Chapter 4]{Perthame2002}, for the sake of the proof of regularity, no specific form for the dissipation measure $q=m+n$ like in \eqref{nonl-dis-eq} is required. This leads to the following broad class of solutions. 

\begin{defn}\label{def:gen_kinetic_sol}  Let $T\ge 0$, $u_0\in L^1(\R^d_x)$ and $S\in L^1([0,T]\times \R^d_{x})$. We say that $f \in L^{\infty}([0,T]\times\R^{d}_x\times\R_v)) \cap L^{\infty}([0,T];L^1(\R^{d}_x\times\R_v))$ is a generalized kinetic solution to \eqref{eq:par-hyp} if 
\begin{enumerate}
\item
For some non-negative measure $\nu$ 
we have
\begin{align*}
    |f(t,x,v)|=\sgn(v)f(t,x,v) &\le 1, \\
    \partial_v f=\delta_{v=0}-\nu(t,x,v).
\end{align*}
\item There is a non-negative measure $q\in\calm^{+}$ such that, 
\[
 \partial_{t}f+F'(v)\cdot\nabla_{x}f-\nonl'(v)Lf=\partial_{v}q+ \nu S\quad\text{on }(0,T)\times\R_{x}^{d}\times\R_{v},
\]
in the sense that, for all $\varphi\in C^\infty_c([0,T)\times\R^d_x\times \R_v)$ we have that
\[
-\int\partial_{t}\varphi f-\nabla_{x} \varphi \cdot F'(v) f-\int L \varphi \nonl'(v)f=\int \varphi(0,x,v) \chi(u_0(x),v)-\int \partial_{v}\varphi q +\int \varphi \nu S.
\]
\item We have
\begin{equation}\label{eq:decay_dissipation_measures}
\int q(t,x,v)\dd x \dd t\le\mu(v)\in L_{0}^{\infty}(\R),
\end{equation}
where $L_{0}^{\infty}$ is the space of $L^{\infty}$-functions vanishing for $|v|\to\infty$.

\end{enumerate}
\end{defn}

We note that analogously to \cite[Remark 4.1.4]{Perthame2002} we have $\nu \in  L^\infty_{t,x}\mathcal{M}_v$ with 
\begin{equation}\label{eq:nu}
\int \nu(t,x,v)\dd v = 1\quad\text{for a.e. } (t,x). 
\end{equation}

\begin{rem}\label{eq:kinetic_form_2}
    For a nonnegative measure $\nu \in \mathcal{M}_{t,x,v}$ we set $[\nu](t,x,v):= \nu(t,x,(-\infty,v])$. Observe that by Fubini's theorem we have $\partial_v[\nu]=\nu$, cf.~Example 1.75 in \cite{AFP00}. Due to Definition \ref{def:gen_kinetic_sol} (i) and (ii), we then have that
    \begin{equation}\begin{split}\label{eq:new_q_form}
     \partial_{t}f+F'(v)\cdot\nabla_{x}f-\nonl'(v)Lf
     &=\partial_{v}q+\nu S =\partial_{v}(q+[\nu] S) =: \partial_{v}\qq
    \end{split} \end{equation}
    and, by \eqref{eq:nu},
    \begin{equation}\begin{split}\label{eq:new_q_form_2}
     \|\qq\|_{L^\infty_v \calm_{TV}} 
     \le \|q\|_{L^\infty_v \calm_{TV}}  + \|[\nu] S\|_{L^\infty_v \calm_{TV}} 
     \le \|q\|_{L^\infty_v \calm_{TV}} + \|S\|_{L^1_{t,x}}.
    \end{split} \end{equation}
    We further have that
    \begin{equation}\begin{split}\label{eq:new_q_form_3}
       \qq -[\delta_{v=0}S] 
       = q+[\nu-\delta_{v=0}]S
       = q-([\partial_v f]S)
       = q-fS.
    \end{split} \end{equation}
\end{rem}

\begin{lem} 
\label{lem:L1=0}
  For $\varphi \in C^2_b(\R^d_x)$ we have 
  \begin{equation}\begin{split}\label{eq:Lphi_bound}
      \|L\varphi\|_{L^\infty_x}
      \lesssim \1_{[1,2)}(a)\|D^2\varphi\|_{L^\infty_x}+     
      \1_{(0,1)}(a)\|D^1\varphi\|_{L^\infty_x}+ \|\varphi\|_{L^\infty_x}.
  \end{split}  \end{equation}
For $R>0$ let $\varphi^R\in\CRci(\R^d_x)$ with $|\varphi^R|\le 1$, $\varphi^R=1$ on $B_R(0)$ and $\supp\varphi^R\subset B_{R+1}(0)$. Then, $L\varphi^R$ is uniformly bounded and $L\varphi^R(x) \to 0$ for all $x\in \R^d$.
\end{lem}
\begin{proof}
  We first prove \eqref{eq:Lphi_bound} by noting that, using Taylor's formula,
  \begin{equation*}\begin{split}
      L\varphi(x)
      &=-\int_{|y|\le 2} \big(\varphi(x+y)-\varphi(x) - \1_{[1,2)}(a)( y\cdot \nabla\varphi(x)\big)k(x,y) \dd y\\
      &-\int_{|y| > 2} \big(\varphi(x+y)-\varphi(x) \big)k(x,y) \dd y \\
      &\lesssim \1_{[1,2)}(a)\|D^2\varphi\|_{L^\infty}\int_{|y|\le 2} |y|^2 k(x,y)\ \dd y +     
      \1_{(0,1)}(a)\|D^1\varphi\|_{L^\infty}\int_{|y|\le 2} |y|\ k(x,y) \dd y \\
      &+\|\varphi\|_{L^\infty}\int_{|y| > 2} k(x,y) \dd y\\
      &\lesssim \|\varphi\|_{C^2} \int_{\R^d_y} ((\1_{[1,2)}(a) |y|^2 +\1_{(0,1)}(a) |y|)\wedge 1)\ k(x,y) \dd y.
  \end{split}  \end{equation*}
  
  In particular, by \eqref{eqn:sec5_as} this implies that $L\varphi^R$ is uniformly bounded. We further have 
  \begin{equation}\begin{split}
    L\varphi^R(x)
    =&-\int_{|x+y|\le R} \big(\varphi^R(x+y)-\varphi^R(x) - \1_{[1,2)}(\frex)\1_{|y|\le 2}(y) y\cdot \nabla\varphi^R(x)\big)k(x,y) \dd y \\
    & -\int_{|x+y| > R} \big(\varphi^R(x+y)-\varphi^R(x) - \1_{[1,2)}(\frex)\1_{|y|\le 2}(y) y\cdot \nabla\varphi^R(x)\big)k(x,y) \dd y.
  \end{split}  \end{equation}
  Therefore, for $r\ge 2$, $R>r$ and $|x|<R-r$ we have
    \begin{equation}\begin{split}
      L\varphi^R(x)
      &\le \int_{|x+y|> R} k(x,y) 
      \le \int_{|y|> r} k(x,y)
      = \int_{|y|> r} (|y|^2 \wedge 1) k(x,y).
    \end{split}  \end{equation}
  Hence, for fixed $r$ and $R\to\infty$ we observe the pointwise convergence to zero. 
\end{proof}

An advantage of the concept of generalized kinetic solutions is that this class of solutions stable/closed in the sense expressed in the following Lemma.
\begin{lem}
    Let $f^n$ be a sequence of generalized kinetic solutions to \eqref{eq:par-hyp} with $u^n_0 \to u_0$ in $L^1(\R^d_x)$ and $S^n \to S$ in $L^1([0,T]\times\R^d_x)$. Suppose that there is a $C>0$ such that $\|f^n\|_{L^{\infty}([0,T];L^1(\R^{d}_x\times\R_v))} \le C$ and that \eqref{eq:decay_dissipation_measures} is satisfied with a uniform $\mu$. Then there is a subsequence $f_{n_k} \rightharpoonup^* f$ in $L^{\infty}_{t,x,v}$, and $f$ is a generalized kinetic solution to \eqref{eq:par-hyp}.
\end{lem}

In the following, we use for $\varphi\in C_c^\infty([0,T))$ the short-hand notation $\overline{g}(v):=\int_{\R_+}\int_{\R^{d}_x} \varphi(r)g(r,x,v) \, \dd x\dd r$ for a generic $g:(r,x,v)\mapsto g(r,x,v)$. 

\begin{lem} \label{lem:kinetic_bounded_eta} Let $f$ be a generalized kinetic solution to \eqref{eq:par-hyp} with $u_{0}\in L^{1}(\R_{x}^{d})$, $S\in L^{1}([0,T]\times\R_{x}^{d})$, let $\varphi\in C_c^\infty([0,T))$, and let $\qq$ be defined as in Remark \ref{eq:kinetic_form_2}. Then, for every $\eta\in C^{\infty}(\R_{v})$ with $\eta'\in C_{c}^{\infty}(\R_{v})$, we have 
\begin{align}
\begin{split}
\int_{\R_v}\eta' (\overline{\qq} - 1_{v>0} \overline{S}) \,\dd v
=&\int_{\R_{v}} \eta  \left(\int_{\R_{x}^{d}}\varphi(0)\chi(u_0(x),v)\,\dd x\right)\dd v\\
&+\int_{\R_+}\int_{\R_{v}}\eta \varphi' \left(\int_{\R_{x}^{d}}f\,\dd x\right)\dd v\dd r+ \eta(0) \overline{S}.
\label{eq:ctn}
\end{split}
\end{align}
\end{lem}
\begin{proof}
  By a standard approximation argument, dominated convergence, and Lemma \ref{lem:L1=0}, the kinetic formulation yields \eqref{eq:ctn}, for every $\eta\in C_{c}^{\infty}(\R_{v})$. Next we prove that \eqref{eq:ctn} continues to hold for all $\eta\in C^{\infty}(\R_{v})$ with $\eta'\in C_{c}^{\infty}(\R_{v})$. For $R>0$ let $\varphi_R\in C_{c}^{\infty}(\R_{v})$ be such that $\varphi_R(v)= 1$ for $|v|\le R$, $\supp\varphi_R\subset [-(R+1),R+1]$ and $|\varphi_R|+|\varphi_R'|\lesssim 1$. 
  Defining $\eta_R:=\eta\varphi_R$, we have by \eqref{eq:ctn}
\begin{align*}
\int_{\R_v}(\eta\varphi_R)' \overline{\qq} \,\dd v
&=\int_{\R_{v}} \eta_R  \left(\int_{\R_{x}^{d}}\varphi(0)\chi(u_0(x),v)\,\dd x\right)\dd v+\int_{\R_{v}}\eta_R\left(\int_{\R_+}\int_{\R_{x}^{d}}\varphi' f\,\dd x\,\dd r\right)\dd v .
\end{align*}
Since $\eta_R$ is uniformly bounded in $R$, $\eta_R\to\eta$ locally uniformly, and since
\begin{equation}\label{eq:l1_v}
  v\mapsto \int_{\R_{x}^{d}}\varphi(0)\chi(u_0(x),v)\,\dd x+\int_{\R_+}\int_{\R_{x}^{d}}\varphi'(r)f(r,x,v)\,\dd x\dd r  \in L^{1}(\R_{v}),
\end{equation}
we may take the limit $R\to\infty$ in the two terms on the right-hand side by dominated convergence. We next note that
\begin{align*}
\int_{\R_v}(\eta\varphi_R)' \overline{\qq} \,\dd v 
&= \int_{\R_v}(\eta\varphi_R)' \overline{(q+[\nu-\delta_{v=0}]S)} \,\dd v + \int_{\R_v}(\eta\varphi_R)' \overline{[\delta_{v=0}]S} \,\dd v \\
&= \int_{\R_v}(\eta\varphi_R)' \overline{(q+[\nu-\delta_{v=0}]S)} \,\dd v - \eta(0) \overline{S}
\end{align*}
and
\begin{align*}
\int_{\R_v}(\eta\varphi_R)' \overline{(q+[\nu-\delta_{v=0}]S)} \,\dd v 
=\int_{\R_v}(\eta'\varphi_R+\eta\varphi_R') \overline{(q+[\nu-\delta_{v=0}]S)} \,\dd v.
\end{align*}
By dominated convergence the contribution from the term $\eta'\varphi_R$ converges, since $\eta'$ has compact support and $\overline{(q+[\nu-\delta_{v=0}]S)} \in L^1_{loc}(\R_v)$. By definition of $[\nu]$, \eqref{eq:nu} and dominated convergence, we have that $\overline{[\nu-\delta_{v=0}]S} \in L_0^\infty(\R_v)$. 
Hence, the second term on the right hand side vanishes for $R\to \infty$, since both $\eta$ and $\varphi_R'$ are bounded with $\supp \varphi_R'\subset[-(R+1),-R]\cup[R,R+1]$. In conclusion, 
\begin{align*}
\int_{\R_v}(\eta\varphi_R)' \overline{\qq} \,\dd v 
&\to \int_{\R_v} \eta' \overline{(q+[\nu-\delta_{v=0}]S)} \,\dd v - \eta(0) \overline{S}\\
&\to \int_{\R_v} \eta' (\overline{\qq} - 1_{v>0} \overline{S}) \,\dd v - \eta(0) \overline{S}. 
\end{align*}
\end{proof}

In \cite{GST20} the following estimate on the kinetic measure $q$ was provided for local PDEs. We carefully need to repeat the argument with additional subtleties caused by the possible lack of continuity in time of generalized kinetic solutions.

\begin{lem} \label{lem:ph-est-1}Let $f$ be a generalized kinetic solution to \eqref{eq:par-hyp} with $u_{0}\in L^{1}(\R_{x}^{d})$, $S\in L^{1}([0,T]\times\R_{x}^{d})$, let $\varphi\in C_c^\infty([0,T))$, and let $\qq$ be defined as in Remark \ref{eq:kinetic_form_2}. Then, the map $v\mapsto  \overline{\qq}(v)$ is continuous and, for all $v_{0}\in\R_{v}$, and we have
\begin{equation}\label{eq:linfty_bound_q}
\begin{split}
\int_{[0,T)}\int_{\R^{d}_x} \qq(r,x,v_0)\,\dd x\dd r
&\le \int_{\R_{x}^{d}}(\sgn(v_{0})(u_{0}-v_{0}))_{+}\,\dd x\,\dd r. 
\end{split}
\end{equation}
\end{lem}
\begin{proof}
We first argue that $\overline{\qq}$ has left and right limits.
Due to \eqref{eq:l1_v}, Lemma \ref{lem:kinetic_bounded_eta} implies $\overline{\qq} \in BV_{loc}(\R_v)$ 
and thus the existence of left and right limits.
\medskip

Assume first $v_0\in \R_+$. Let $\phi_{\pm}\in C_{c}^{\infty}(\R_{v})$ with $\phi_{\pm}\ge 0$, $\supp \phi_{+}\subset[0,1]$, $\supp \phi_{-}\subset[-1,0]$, $\int_{\R_v}\phi_{\pm} \dd v=1$ and define $\phi_{\pm}^\eps(v)=\eps^{-1}\phi_{\pm}(\eps^{-1}v)$ for $\eps>0$. Moreover let $\eta_{\pm}^\eps$ be such that $(\eta_{\pm}^\eps)'(v)=\phi_{\pm}^\eps(v-v_0)$ and $(\eta_{\pm}^\eps)(v_0)=0$. Observe that $(\eta_{\pm}^{\eps})'\to\delta_{v=v_0}$ and $\eta_{\pm}^{\eps}(v)\to \sgn_{+}(v-v_0)$ as $\eps\searrow 0$ independent of the choice of $\pm$. Choosing now $\eta:=\eta_{\pm}^\eps$ in \eqref{eq:ctn} and using dominated convergence to take the limit $\eps\searrow 0$, we obtain 
\begin{equation}\begin{split}\label{eq:q-id}
  \overline{\qq}(v_{0}\pm) & = \int_{\R_{x}^{d}}\varphi(0)(u_0-v_0)_+\,\dd x+\int_{\R_+} \varphi' \left(\int_{\R_{x}^{d}}\int_{\R_{v}}\sgn_+(v-v_0) f\,\dd v\dd x\right)\dd r+  \overline{S}.
\end{split}\end{equation}
In particular, we deduce that $\overline{\qq}(v_{0}-)=\overline{\qq}(v_{0}+)$, so that $\overline{\qq}$ is continuous on $\R_+$. The case $v_{0}\in\R_{-}$ is treated analogously replacing the conditions $\phi_{\pm}\ge 0$ and $\int_{\R_v}\phi_{\pm} \dd v=1$ by $\phi_{\pm}\le 0$ and $\int_{\R_v}\phi_{\pm} \dd v=-1$, respectively, so that $\eta_{\pm}^{\eps}(v)\to \sgn_{+}(v-v_0)$ is replaced by $\eta_{\pm}^{\eps}(v)\to \sgn_{-}(v-v_{0})$. We obtain that
\begin{equation*}
\begin{split}
&\int_{\R_+}\int_{\R^{d}_x}\varphi(r) \qq(r,x,v_0)\,\dd x\dd r \\
&\le
\int_{\R_{x}^{d}}\varphi(0)(\sgn(v_{0})(u_0-v_0))_+\,\dd x
+\int_{\R_+} \varphi' \left(\int_{\R_{x}^{d}}\int_{\R_{v}}\sgn_{\sgn(v_{0})}(v-v_0) f\,\dd v\dd x\right)\dd r+\overline{S}.
\end{split}
\end{equation*}
Note that if $\varphi' \le 0$ then the second term on the right hand side has non-negative sign, due to Definition \ref{def:gen_kinetic_sol}, (iii). Choosing $\varphi$ to approximate $1_{[0,T)}$, and using Fatou's Lemma and dominated convergence we deduce \eqref{eq:linfty_bound_q}.
\end{proof}

\begin{lem} \label{lem:ph-est-1a}Let $f$ be a generalized kinetic solution to \eqref{eq:par-hyp} with $u_{0}\in L^{1}(\R_{x}^{d})$, $S\in L^{1}([0,T]\times\R_{x}^{d})$. 
Then,  
\begin{equation}\label{eq:f_est}
\begin{split}
\|f\|_{L^{\infty}([0,T];L^1(\R^{d}_x\times\R_v))} \le \|u_0\|_{L^1_x}+3\|S\|_{L^1_{t,x}}.
\end{split}
\end{equation}
\end{lem}
\begin{proof}
   Adding \eqref{eq:q-id} for $v_0=0$ and the analogous equality for $\sgn_-$ we have that
   \begin{align}
\begin{split}
-\int_{\R_+}\varphi' \left(\int_{\R_{x}^{d}}\int_{\R_{v}} |f|\,\dd x\dd v\right)\dd r
&\le \int_{\R_{x}^{d}}\varphi(0)|u_0|(x)\,\dd x-2\int_{\R_+}\int_{\R_{x}^{d}}\int_{\R_v}\varphi\qq+ \overline{S},
\end{split}
\end{align}
and we conclude since $-\qq \le  -[\nu] S \le |S|$ by  \eqref{eq:nu}. 
\end{proof}

%
\section{Application to nonlocal degenerate PDEs
}\label{App_PME}

In this section, we provide the proof of Theorem \ref{cor:pme_l1} by applying the averaging lemmata obtained in Section \ref{AvLem} to generalized kinetic solutions and to entropy solutions to \eqref{pme_sys}.
The idea is in essence to interpolate the outcomes of Lemma \ref{lem:av} and Lemma \ref{lem:av3} via Lemma \ref{lem:interpol}.

\begin{proof}[Proof of Theorem \ref{cor:pme_l1}]
 We extend $\f$ to all times $t\in \R$ by multiplying with a smooth cut-off function $\vp\in \CRci(0,T)$ with $0\le \vp\le 1$ and $\|\vp'\|_{\LR{1}_t}\le 2$.
 We write $[\nu](t,x,v):=\nu(t,x,(-\infty,v])$ and $\qq:=q+[\nu]S$ as in Remark \ref{eq:kinetic_form_2}.
 This gives rise to the distributional equation
  \begin{align}
   \partial_t(\vp\f)-\nonl'L_x(\vp\f) &= \partial_\vf (\vp \qq) + \vp' f.
  \end{align}
 Set $g_0:= \vp' f$ and $g_1:=\vp \qq$.
 Since $f$ is a generalized kinetic solution, we have $\vp\f\in \LR{\infty}(\R_t\times \R^d_x\times \R_v)$ with $\|\f\|_{\LR{\infty}_{t,x,v}}\le 1$ as well as $|g_0|\in \calm_{TV}(\R_t\times\R^d_x\times\R_v)$, and by \eqref{eq:new_q_form_2} we have $g_1\in \LR{\infty}(\R_{v};\calm_{TV}(\R_{t}\times\R_{x}^{d}))$.
 Therefore we are in the situation of Lemma \ref{lem:av} and Lemma \ref{lem:av3} with $p(x,\xi)$ given by \eqref{fourier-repr-ker}.
 
 \newcounter{cor:pme_l1_prf} 
  \refstepcounter{cor:pme_l1_prf} 
  \textit{Step} \arabic{cor:pme_l1_prf}\label{cor:pme_l1_prf_st1}. In this step we establish for $s_x\in (0,\frac{\frex}{m})$ the bound
  \begin{align}\label{js02}
  \norm{\vp u^\eta}_{S^{(0,s_x)}_{m,\infty}B}^m\lesssim (1+\|\eta'\|_{\LR{1}_v})(\|u_0\|_{L^{1}_{x}}^m + \|S\|_{L^1_{t,x}}^m
  + 1) + \|\eta\nonl' \vp f\|_{\LR{1}_{t,x,v}}.
 \end{align}
 By Lemma \ref{lem:av}, we have for any $s_x\in(0,\frac{\frex}{m})$
 \begin{align}\label{eq:bound_1}
 \begin{split}
  \norm{\vp u^\eta}_{S^{(0,s_x)}_{m,\infty}B}^m
  &\lesssim \|g_{0}\|_{\calm_{TV}}+(1+\|\eta'\|_{\LR{1}_v})\|g_{1}\|_{\LR{\infty}_v\calm_{TV}} \\
  &\quad + \|\vp u\|_{\LR{\infty}_t\LR{1}_{x}\cap\LR{1}_{t,x}}^m + \|\vp u\|_{\LR{1}_{t,x}} + \|\eta\nonl'\vp f\|_{\LR{1}_{t,x,v}}.
 \end{split}
 \end{align}
 Since $[0,T]$ is compact, we have by Lemma \ref{lem:ph-est-1a}
\begin{align*}
 \|\vp u\|_{\LR{\infty}_t\LR{1}_{x}\cap\LR{1}_{t,x}}\lesssim \|\vp u\|_{\LR{\infty}_t\LR{1}_{x}} \lesssim \|u_0\|_{L^{1}_{x}}+ \|S\|_{L^1_{t,x}},
\end{align*}
and similarly
 \begin{align*}
 \|g_0\|_{\calm_{TV}}&=\|\vp' f \|_{\calm_{TV}} \lesssim   \|u\|_{L^\infty_{t} L^1_{x}} \|\vp'\|_{L^1_t} \le 2\|f\|_{L^\infty_{t} L^1_{x,v}}\lesssim \|u_0\|_{L^{1}_{x}}+ \|S\|_{L^1_{t,x}}.
\end{align*}
This estimates the first, the third, and the fourth term on the right-hand side of \eqref{eq:bound_1}. For the second term we utilize \eqref{eq:new_q_form_2} and Lemma \ref{lem:ph-est-1}, and obtain
\begin{align*}
 \|g_1\|_{\LR{\infty}_v\calm_{TV}}
 \lesssim \|\qq\|_{\LR{\infty}_v\calm_{TV}}
 \lesssim \|q\|_{\LR{\infty}_v\calm_{TV}}+\|S\|_{\LR{1}_{t,x}}
 \lesssim \|u_0\|_{L^{1}_{x}}+\|S\|_{L^1_{t,x}}.
\end{align*}
Hence, \eqref{js02} follows since $\|u_0\|_{L^{1}_{x}}+ \|S\|_{L^1_{t,x}} + (\|u_0\|_{L^{1}_{x}}+ \|S\|_{L^1_{t,x}})^m\lesssim \|u_0\|_{L^{1}_{x}}^m+ \|S\|_{L^1_{t,x}}^m + 1$.
 
 \refstepcounter{cor:pme_l1_prf} 
  \textit{Step} \arabic{cor:pme_l1_prf}\label{cor:pme_l1_prf_st2}. In this step we establish the bound
  \begin{align}\label{js03}
  \norm{\vp u^\eta}_{S^{(1,0)}_{1,\infty}B}\lesssim (1+\|\eta'\|_{\LR{1}_v})(\|u_0\|_{L^{1}_{x}}+ \|S\|_{L^1_{t,x}})
  + \|\eta\nonl' \vp f\|_{\LR{1}_{t,x,v}}.
 \end{align}
 Indeed, by Lemma \ref{lem:av3}, we have
 \begin{align*}
  \norm{\vp u^\eta}_{S^{(1,0)}_{1,\infty}B}\lesssim \|g_{0}\|_{\calm_{TV}} +(1+\|\eta'\|_{\LR{1}_v})\|g_{1}\|_{\LR{\infty}_v\calm_{TV}} + \|\vp u\|_{\LR{1}_{t,x}} + \|\eta\nonl'\vp f\|_{\LR{1}_{t,x,v}},
 \end{align*}
 and the first three terms on the right-hand side are estimated as in Step \ref{cor:pme_l1_prf_st1}. 
  
   \refstepcounter{cor:pme_l1_prf} 
  \textit{Step} \arabic{cor:pme_l1_prf}\label{cor:pme_l1_prf_st3}. In this step we obtain the desired regularity by interpolation.
  That is, we will show for $p\in (1,m)$
  \begin{align}\label{cor:ome_l1_e1}
       \|\varphi u^\eta\|_{\WSR{\sigma_t}{p}(\R;\WSR{\sigma_x}{p}(\R^d))}^p
       &\lesssim (1+\|\eta'\|_{\LR{1}_v})(\|u_0\|_{L^{1}_{x}}^m + \|S\|_{L^1_{t,x}}^m + 1) +  \|\eta\nonl' \vp f\|_{\LR{1}_{t,x,v}}.
  \end{align}
   Note that $S^{\overline{r}}_{p,\infty}B\subset S^{\overline{r}-\overline{\eps}}_{p,1}B$ for
   $\overline{\eps}\in (0,\infty)^2$ by Remark \ref{rem:spaces}.
   Using Lemma \ref{lem:interpol} with $\overline{r}_0=(1-\eps,-\eps)$, $\overline{r}_1=(-\eps,\frac{\frex}{m}-\eps)$, $p_0=1$, $p_1=m$, $q_0=q_1=1$ and $\theta=\frac{p-1}{p}\frac{m}{m-1}$, so that $\overline{r}_\theta=(\kappa_t-\eps, \kappa_x-\eps)$, $p_\theta=p$ and $q_\theta=1$, we obtain with $s_x\in (\frac{\frex}{m}-\eps,\frac{\frex}{m})$
   \begin{align*}
       \|\varphi u^\eta\|_{S^{(\kappa_t-\eps,\kappa_x-\eps)}_{p,1}B}^p&\lesssim \|\varphi u^\eta\|_{S^{(-\eps,\frac{\frex}{m}-\eps)}_{m,1}B}^{p(1-\theta)}\|\varphi u^\eta\|_{S^{(1-\eps,-\eps)}_{1,1}B}^{p\theta}
       \lesssim \|\varphi u^\eta\|_{S^{(0,s_x)}_{m,\infty}B}^{p(1-\theta)}\|\varphi u^\eta\|_{S^{(1,0)}_{1,\infty}B}^{p\theta} \\
       &\lesssim (1+\|\eta'\|_{\LR{1}_v})(\|u_0\|_{L^{1}_{x}}^m + \|S\|_{L^1_{t,x}}^m + 1) +  \|\eta\nonl' \vp f\|_{\LR{1}_{t,x,v}}.
   \end{align*}
   Here we have used $c^{p(1-\theta)}(c^\frac1m)^{p\theta}=c$ with $c:=(1+\|\eta'\|_{\LR{1}_v})(\|u_0\|_{L^{1}_{x}}^m + \|S\|_{L^1_{t,x}}^m + 1) +  \|\eta\nonl' \vp f\|_{\LR{1}_{t,x,v}}$.
   Since \[S^{(r_t,r_x)}_{p,1}B\subset B^{r_t-\eps}_{p,p}(\R;B^{r_x-\eps}_{p,p}(\R^d_x))= \WSR{r_t-\eps}{p}(\R;\WSR{r_x-\eps}{p}(\R^d))\] for $r_t,r_x>\eps>0$, see Proposition 2.2.3/2 and Remark 2.3.4/4 in \cite{ScT87}, we obtain the claimed estimate \eqref{cor:ome_l1_e1}.
   
   \medskip
   
 We now show an additional estimate in the case $p=m$.
 Choose $\eps\in (0,m)$ sufficiently small such that for $\tilde p:=m-\eps$ with corresponding $\tilde\kappa_x:=\frac{\tilde p-1}{\tilde p}\frac{\frex}{m-1}$ there holds $\WSR{\tilde\sigma_x}{\td p}(\R^d)\subset \WSR{\sigma_x}{m}(\R^d)$ for some $\tilde\sigma_x\in (0,\tilde\kappa_x)$.
 Since also
 \begin{align*}
     \tilde\kappa_t - \frac{1}{\tilde p}>-\frac{1}{m} \ \Leftrightarrow \ \frac{\eps}{m-\eps}\frac{1}{m-1}-\frac{1}{m-\eps}>-\frac{1}{m} \ \Leftrightarrow \eps>0,
 \end{align*}
 there exists some $\tilde\sigma_t\in (0,\tilde\kappa_t)$ with the same inequality $\tilde\sigma_t - \frac{1}{\tilde p}>-\frac{1}{m}$, and thus we have for any Banach space $E$ the Sobolev embedding $\WSR{\tilde\sigma_t}{\tilde p}(\R;E)\subset \LR{m}(\R;E)$.
 In total we obtain
 \begin{align}
 \begin{split}\label{js02a}
     \norm{\vp u^\eta}_{\LR{m}(\R;\WSR{\sigma_x}{m}(\R^d))}^{\tilde p}&\le \norm{\vp u^\eta}_{\WSR{\tilde\sigma_t}{\tilde p}(\R;\WSR{\tilde\sigma_x}{\tilde p}(\R^d))}^{\tilde p} \\
     &\lesssim (1+\|\eta'\|_{\LR{1}_v})(\|u_0\|_{L^{1}_{x}}^m + \|S\|_{L^1_{t,x}}^m + 1) +  \|\eta\nonl' \vp f\|_{\LR{1}_{t,x,v}}.
 \end{split}
 \end{align} 

  \refstepcounter{cor:pme_l1_prf} 
  \textit{Step} \arabic{cor:pme_l1_prf}\label{cor:pme_l1_prf_st5}. 
  In this step we show that under the additional assumption $f(t,x,v)=\chi(u(t,x),v)$ for a.a. $(t,x,v)$ and $|\nonl|(u)\lesssim |u|+|u|^m$, we have for all $p\in (1,m]$ and $\eps>0$
  \begin{align}\label{js06}
    \|\varphi u^\eta\|_{\WSR{\sigma_t}{p}(\R;\WSR{\sigma_x}{p}(\R^d))}^p
       &\lesssim (1+\|\eta'\|_{\LR{1}_v})^{1+\eps}(\|u_0\|_{L^{1}_{x}}^{m+\eps} + \|S\|_{L^1_{t,x}}^{m+\eps} + 1).
  \end{align}
  Observe that
  \begin{align*}
    \|\eta\nonl' \vp f\|_{\LR{1}_{t,x,v}}= \|\nonl(\vp u^\eta)\|_{\LR{1}_{t,x}} \lesssim \|\vp u^\eta\|_{\LR{1}_{t,x}}+ \|\vp u^\eta\|_{\LR{m}_{t,x}}^m.
  \end{align*}
  By virtue of Lemma \ref{lem:ph-est-1a} we have
  \begin{align*}
   \|\vp u^\eta\|_{\LR{1}_{t,x}}\lesssim \|u^\eta\|_{\LR{\infty}_t\LR{1}_{x}}\lesssim (1+\|\eta'\|_{\LR{1}_v})(\|u_0\|_{\LR{1}_x} + \|S\|_{\LR{1}_{t,x}})\le (1+\|\eta'\|_{\LR{1}_v})(\|u_0\|_{L^{1}_{x}}^{m+\eps} + \|S\|_{L^1_{t,x}}^{m+\eps}+ 1),   
  \end{align*}
  so that by \eqref{cor:ome_l1_e1} (respectively by \eqref{js02a} in the case $p=m$) it suffices to show that $\|\vp u^\eta\|_{\LR{m}_{t,x}}^m\lesssim (1+\|\eta'\|_{\LR{1}_v})(\|u_0\|_{L^{1}_{x}}^{m+\eps} + \|S\|_{L^1_{t,x}}^{m+\eps}+ 1)$.
  For the purpose of this proof, we hence may assume
  \begin{align*}
      (1+\|\eta'\|_{\LR{1}_v})(\|u_0\|_{L^{1}_{x}}^{m+\eps} + \|S\|_{L^1_{t,x}}^{m+\eps}+ 1)\le \|\vp u^\eta\|_{\LR{m}_{t,x}}^m,
  \end{align*}
  since otherwise there is nothing to prove.
  For $s_x\in (0,\frac{\frex}{m})$ we then have by Steps \ref{cor:pme_l1_prf_st1} and \ref{cor:pme_l1_prf_st2}
  \begin{align*}
      \norm{\vp u^\eta}_{S^{0,s_x}_{m,\infty}}^m&\lesssim (1+\|\eta'\|_{\LR{1}_v})(\|u_0\|_{L^{1}_{x}}^m + \|S\|_{L^1_{t,x}}^m + 1) + \|\eta\nonl' \vp f\|_{\LR{1}_{t,x,v}} \lesssim \|\vp u^\eta\|_{\LR{m}_{t,x}}^m. \\
      \norm{\vp u^\eta}_{S^{(1,0)}_{1,\infty}B}&\lesssim (1+\|\eta'\|_{\LR{1}_v})(\|u_0\|_{L^{1}_{x}}+ \|S\|_{L^1_{t,x}}) + \|\eta\nonl' \vp f\|_{\LR{1}_{t,x,v}}\lesssim \|\vp u^\eta\|_{\LR{m}_{t,x}}^m.
  \end{align*}
  We now choose $\eps_x\in (0,s_x)$ and $\theta\in (0,1)$ subject to
  \begin{align*}
      \frac{d(m-1)}{d(m-1)+m(s_x-\eps_x)} < \theta.
  \end{align*}
  Then for $p:=\frac{m}{m(1-\theta)+\theta}\in (1,m)$ we have $\frac1{p}=1-\theta+\frac{\theta}{m}$.
  Thus, for $\eps_{t}\in (0,1-\theta)$ and $\overline{r}:=(1-\theta-\eps_{t},\theta (s_x -\eps_x))$ we have as in Steps \ref{cor:pme_l1_prf_st3} and \ref{cor:pme_l1_prf_st4}
  \begin{align*}
      \|\vp u^\eta\|_{\WSR{r_t}{p}(\R;\WSR{r_x}{p}(\R^d))} \lesssim \|\vp u^\eta\|_{S^{(1,0)}_{1,\infty}B}^{1-\theta}\|\vp u^\eta\|_{S^{(0,s_x)}_{m,\infty}B}^\theta\lesssim \|\vp u^\eta\|_{\LR{m}_{t,x}}^{m(1-\theta)+\theta}.
  \end{align*}
  Next we set $p^*:=\frac{dp}{d-p r_x}$.
  Choosing $\eps_x\in (0,s_x)$ sufficiently close to $s_x$ and hence $r_x$ sufficiently small, we can always guarantee that $p^*<\infty$.
  Thus, applying the Sobolev embedding in $x$ and a trivial embedding in $t$, we obtain
  \begin{align*}
      \|\vp u^\eta\|_{\LR{p}_t\LR{p^*}_x}\lesssim \|\vp u^\eta\|_{\WSR{r_t}{p}(\R;\WSR{r_x}{p}(\R^d))}\lesssim \|\vp u^\eta\|_{\LR{m}_{t,x}}^{m(1-\theta)+\theta}.
  \end{align*}
  Moreover $p^*\in (m,\infty)$ since
  \begin{align*}
   p^*>m \quad \Leftrightarrow \quad \frac{d(m-1)}{d(m-1)+m(s_x-\eps_x)} < \theta. 
  \end{align*}
  Hence, choosing $\tau:=\frac{m-1}{m}\frac{p^*}{p^*-1}\in (0,1)$ and $m^*:=\tau m$, we have
  \begin{align*}
      \frac1m=\frac{1-\tau}{\infty}+\frac{\tau}{m^*}, \quad \text{and} \quad \frac1m=\frac{1-\tau}{1}+\frac{\tau}{p^*},
  \end{align*}
  so that
  \begin{align*}
      \|\vp u^\eta\|_{\LR{m}_{t,x}}^m\lesssim \|\vp u^\eta\|_{\LR{\infty}_t\LR{1}_x}^{(1-\tau)m}\|\vp u^\eta\|_{\LR{m^*}_t\LR{p^*}_x}^{\tau m}
  \end{align*}
  Since
  \begin{align*}
      p>m^* \quad \Leftrightarrow \quad \frac{1}{m-\theta(m-1)}>\frac{1}{\theta+\frac{\theta m}{d(m-1)}(s_x-\eps_x)}
  \end{align*}
  and the right condition is fulfilled as $\theta\nearrow 1$, we find a sufficiently large $\theta<1$ such that $p>m^*$.
  Thus
  \begin{align*}
      \|\vp u^\eta\|_{\LR{m}_{t,x}}^m\lesssim \|\vp u^\eta\|_{\LR{\infty}_t\LR{1}_x}^{(1-\tau)m}\|\vp u^\eta\|_{\LR{p}_t\LR{p^*}_x}^{\tau m} \lesssim \|\vp u^\eta\|_{\LR{\infty}_t\LR{1}_x}^{(1-\tau)m}\|\vp u^\eta\|_{\LR{m}_{t,x}}^{(m(1-\theta)+\theta)\tau m}.
  \end{align*}
  Since $p>m^*=\tau m$ is equivalent to $m>(m(1-\theta)+\theta)\tau m$, this implies
  \begin{align*}
      \|\vp u^\eta\|_{\LR{m}_{t,x}}\lesssim \|\vp u^{\eta}\|_{\LR{\infty}_t\LR{1}_x}^{\frac{1-\tau}{1-(m(1-\theta)+\theta)\tau}}\lesssim \left((1+\|\eta'\|_{\LR{1}_v})(\|u_0\|_{\LR{1}_x}+\|S\|_{\LR{1}_{t,x}})\right)^{\frac{1-\tau}{1-(m(1-\theta)+\theta)\tau}}.
  \end{align*}
  Now ${\frac{1-\tau}{1-(m(1-\theta)+\theta)\tau}}\searrow 1$ as $\theta\nearrow 1$, so that for any $\eps>0$ we find $\theta$ such that the exponent is bounded by $1+\eps$.
  This shows \eqref{js06}.
  
  \refstepcounter{cor:pme_l1_prf} 
  \textit{Step} \arabic{cor:pme_l1_prf}\label{cor:pme_l1_prf_st4}. In this step we remove the localization in time.
  Let $\vp_n(t)=\psi(nt)-\psi(nt-T/2)$, where $\psi\in\CRi(\R)$ with $0\le \psi\le 1$, $\supp\psi\subset(0,\infty)$, $\psi(t)=1$ for $t>T/2$ and $\|\psi'\|_{L^1}=1$. For $n\to\infty$, $\vp_n$ converges to $1_{[0,T]}$ in the supremum norm.
  Hence, we may conclude \eqref{cor:pme_est1_l1} from \eqref{cor:ome_l1_e1} by lower semi-continuity.
  
  \medskip
  
  If we additionally we choose $\eta=1$ on $B_1(0)$, $\supp \eta\subset B_2(0)$ and $0\le \eta\le 1$ and define $\eta_R(v):=\eta(v/R)$, then $\eta_R\to 1$ and $\|\eta_R'\|_{\LR{1}_v}\lesssim 1$ for $R\to \infty$.
  Therefore we may also conclude \eqref{cor:pme_est1_l1_2} from \eqref{js06} by lower semi-continuity.
\end{proof}

\section{Existence of generalized kinetic solutions}

In this section we prove Theorem \ref{prop:wp-kinetic} on the existence of generalized kinetic solutions.

\begin{proof}[Proof of Theorem \ref{prop:wp-kinetic}]\textit{Step 1: Viscous approximation. }
   In this step, we consider the ${\varepsilon_1} \to0$ limit for $\varepsilon_2,\varepsilon_3>0$ fixed and $u_0$, $S$, bounded smooth of solutions to the approximate equation
    \begin{equation}\label{eq:approx}
    \partial_{t}u^{{\varepsilon_1,\varepsilon_2,\varepsilon_3}  }+L^{\varepsilon_1} \nonl^{\varepsilon_3}(u^{{\varepsilon_1,\varepsilon_2,\varepsilon_3}  }) = \varepsilon_2 \Delta  u^{{\varepsilon_1,\varepsilon_2,\varepsilon_3}  }+S \text{ on }(0,T)\times\R_{x}^{d},
    \end{equation}
    with $u^{{\varepsilon_1,\varepsilon_2,\varepsilon_3}  }(0)=u_0$, and $\nonl^{\varepsilon_3}(u):=\nonl(u)+{\varepsilon_3} u$.
    Here, $L^{\varepsilon_1} $ is defined from $L$ by regularizing the kernel $k$, so that $L^{\varepsilon_1} $ becomes a regular convolution operator. 
    
    Since $\varepsilon_2,\varepsilon_3>0$ are fixed in this step we drop them in the notation. 
    For each ${\varepsilon_1} >0$, by \cite{L1996} there is a unique solution to \eqref{eq:approx}. Following the same arguments as the ones leading to Definition \ref{def:kinetic_sol-1}, the kinetic function $f^{{\varepsilon_1} }(t,x,v):=\chi(u^{{\varepsilon_1} }(t,x),v)$ satisfies, in the sense of distributions, the kinetic form
    \begin{equation}\label{eqn:approx_kinetic}
      \partial_{t}f^{{\varepsilon_1} }+\nonl'(v)L^{{\varepsilon_1} }f^{{\varepsilon_1} }= \varepsilon_2 \Delta   f^{{\varepsilon_1} }+\partial_{v}q^{{\varepsilon_1} }+\delta_{v=u^{{\varepsilon_1} }}S
    \end{equation}
    on $(0,T)\times\R_{x}^{d}\times\R_{v}$, where $q^{{\varepsilon_1} }=m^{{\varepsilon_1} }+n^{{\varepsilon_1} }$, $n^{{\varepsilon_1} }$ is given by \eqref{nonl-dis-eq} and $m^{{\varepsilon_1} } =\varepsilon_2 \delta_{v=u^{{\varepsilon_1} }}|\nabla u^{{\varepsilon_1} }|^2$.     Moreover, by direct computation
    \begin{equation}\begin{split}\label{eq:f_eps_del}
        |f^{{\varepsilon_1} }(t,x,v)|=\sgn(v)f^{{\varepsilon_1} }(t,x,v) &\le 1, \\
        \partial_v f^{{\varepsilon_1} }=\delta_{v=0}-\delta_{v=u^{{\varepsilon_1} }}.
    \end{split}\end{equation}
    Choosing, as in \cite[Equation (2.7)]{ChP03}, $s(u)=(u-v)_+$ for $v\ge 0$ and $s(u)=(u-v)_-$ for $v \le 0$, and a standard approximation argument implies \eqref{eq:decay_dissipation_measures}, for $q^{{\varepsilon_1} }$ with the uniform bound
    \begin{equation}\label{eq:mu}
      \mu(v) = 1_{v>0}\|(u_0-v)_+\|_{L^1(\R^d)}+1_{v<0}\|(u_0-v)_-\|_{L^1(\R^d)}.
    \end{equation}

    By a standard approximation argument, the entropy inequalities
    \eqref{entropy-sol-ineq} imply the uniform $L^p$ estimates, for each $p\ge 1$,
    \begin{equation}\begin{split}\label{eq:lp}
      \sup_{t\in[0,T]}\|u^{{\varepsilon_1} }(t)\|_{L^p(\R^d_x)}^p + (p-1) \int_0^T\int_{\R^d} \int_\R v^{p-2} n^{\varepsilon_1} (t,x,v) \dd v \dd x \dd t \\
      \lesssim \|u_0\|_{L^p(\R^d_x)}^p + \|S\|_{L^p([0,T]\times \R^d_x)}^p,
    \end{split}\end{equation}
    and
    \begin{equation}\label{eq:h1}
     \varepsilon_2\int_0^T \int_{\R^d_x} |\nabla u^{{\varepsilon_1} }(t)|^2 dxdt \lesssim \|u_0\|_{L^2(\R^d_x)}^2 + \|S\|_{L^2([0,T]\times \R^d_x)}^2.
    \end{equation}
    
    Again arguing via an approximation argument, for two solutions $u^{1,{\varepsilon_1} }, u^{2,{\varepsilon_1} }$ to \eqref{eq:approx}  with smooth data $u^{1}_0$, $u^{2}_0$ and $S^1$, $S^2$ respectively,  we have that
    \begin{equation}\label{eq:l1_approx}
       \esssup_{t\in[0,T]}\|u^{1,{\varepsilon_1} }(t)-u^{2,{\varepsilon_1} }(t)\|_{L^1_x} \le \|u^{1}_0-u^{2}_0\|_{L^1_x} + \|S^{1}-S^{2}\|_{L^1_{t,x}}.
    \end{equation}

    Using \eqref{eq:h1}, the Aubin-Lions-Simon lemma \cite{Si1986} implies that there is a subsequence (not relabelled) so that  $u^{{\varepsilon_1} } \to u$ in $L^1_tL^1_x$ and $f^{{\varepsilon_1} } \rightharpoonup^* f$ in $L^\infty_{t,x,v}$.
    Due to \eqref{eq:mu} we have $q^{{\varepsilon_1} } \rightharpoonup^* q$ for some $q\ge 0$. The strong convergence of $u^{{\varepsilon_1} }$ implies that for a.e.\ $(t,x)$, for all $v\in\R_v$ we have that $f=\chi(u,v)$, and that \eqref{eq:l1_approx} is satisfied for the respective limits $u^{1}$, $u^{2}$.
    This allows to pass to the limit in \eqref{eq:f_eps_del} to deduce the analogous properties for $f$. 
    
    Taking the limit in \eqref{eq:decay_dissipation_measures} for $q^{{\varepsilon_1} }$ with \eqref{eq:mu}, yields  \eqref{eq:decay_dissipation_measures}  for $q$ with \eqref{eq:mu}. Moreover, using Fatou's Lemma,    with $n$ defined as in  \eqref{nonl-dis-eq}, we have
    \begin{equation}\label{eqn:approx_kinetic_measure_bound}
      q \ge n.
    \end{equation}
    
      Since, for smooth test functions $L^{\varepsilon_1} \varphi$, $L\varphi$ are uniformly bounded and $L^{\varepsilon_1} \varphi \to  L\varphi$ pointwise, we can pass to the limit in \eqref{eqn:approx_kinetic} to obtain that 
    \begin{equation}\label{eqn:approx_kinetic_2}
      \partial_{t}f+\nonl'(v)Lf=\varepsilon_2 \Delta f+\partial_{v}q+\delta_{v=u}S.
    \end{equation}
    In conclusion, $f$ is a generalized kinetic solution to 
    \begin{equation}\label{eq:approx_2}
    \partial_{t}u+L \nonl(u) = \varepsilon_2 \Delta u+S \text{ on }(0,T)\times\R_{x}^{d}.
    \end{equation}

        \textit{Step 2: Vanishing viscosity. } In this step, we consider the limit $\varepsilon_2 \to 0$.
        Since $\varepsilon_3$ is fixed we drop it in the notation.
        Due to $(\nonl^{\varepsilon_3})'\ge {\varepsilon_3} $ we have that
\begin{equation}\begin{split}\label{nonl-dis-eq-2}
n^{\varepsilon_2}(t,x,v)
&=\int_{\R^d} |\nonl^{\varepsilon_3}(u(t,x+y))-\nonl^{\varepsilon_3}(v)| \1_{\mathrm{conv}\set{\tau_y u(t,x),u(t,x)}}(v) k(x,y) \dd y \\
&\ge {\varepsilon_3}  \int_{\R^d} |u(t,x+y)-v| \1_{\mathrm{conv}\set{\tau_y u(t,x),u(t,x)}}(v) k(x,y) \dd y,
\end{split}\end{equation}
and from \eqref{eq:lp} with $p=2$, it follows that
\begin{equation}\begin{split}
  {\varepsilon_3}  &\|u\|_{L^2([0,T];W^{a/2,2}_x)}^2 \\
 & \lesssim 
 \int_0^T\int_{\R^d} \int_\R \int_{\R^d} {\varepsilon_3}  |u(t,x+y)-v| \1_{\mathrm{conv}\set{\tau_y u(t,x),u(t,x)}}(v) k(x,y) \dd y  \dd v \dd x \dd t \\
 &\le \|u_0\|_{L^2(\R^d_x)}^2 + \|S\|_{L^2([0,T]\times \R^d_x)}^2.
\end{split}\end{equation}
Using the Aubin-Lions-Simon lemma \cite{Si1986} we can extract $u^{{\varepsilon_2} } \to u$ in $L^1_tL^1_{x,\textrm{loc}}$ and $f^{{\varepsilon_2} } \rightharpoonup^* f$ in $L^\infty_{t,x,v}$. We may then argue as in Step 1 to take the limit $\varepsilon_2 \to 0$ to obtain a generalized kinetic solution to
\begin{equation}\label{eq:approx_3}
    \partial_{t}u+L \nonl (u) = S \text{ on }(0,T)\times\R_{x}^{d}.
    \end{equation}    
Notably, the estimates \eqref{eq:lp} and \eqref{eq:l1_approx} carry over to these limits.

    \textit{Step 3: Vanishing nonlocal viscosity. }
In this step, we consider the $\varepsilon_3\to0$ limit for $u_0$, $S$, bounded smooth, of the solutions $u^{\varepsilon_3}$ to \eqref{eq:approx_3} constructed in Step 2.
Compared to the ${\varepsilon_1} \to0$ limit of Step 1, in order to pass to the limit $\varepsilon_3\to0$ we need a uniform in $\varepsilon_3$ estimate replacing \eqref{eq:h1}.
For this we invoke Theorem \ref{cor:pme_l1} after localizing $f^{\varepsilon_3}$.

We note that $\nonl^{\eps_3}$ fulfills \eqref{ass_nonl_1} with a uniform constant in $\eps_3>0$, since $\set{\nonl'(v)+\eps_3\le \delta}\subset \set{\nonl'(v)\le \delta}$ and $\#\set{\nonl'(v)+\eps_3 =  \delta} = \#\set{\nonl'(v) = \delta-\eps_3} \le C$.

Let $\eta \in C^\infty_c (\R_v)$ be a cutoff function with $|\eta|\le 1$. Then, by Theorem \ref{cor:pme_l1},
$u^{\eta,\eps_3}=\int_v \eta f^{\eps_3}$ satisfies
     \begin{align*}
      \|u^{\eta,\eps_3}\|_{\WSR{\sigma_t}{p}(0,T;\WSR{\sigma_x}{p}(\R^d))}
      &\lesssim (1+\|\eta'\|_{\LR{1}_v})(\|u_0\|_{L^1_{x}}^m +\|S\|_{L^1_{t,x}}^m + 1) +\|\eta\nonl' f^{\eps_3}\|_{\LR{1}_{t,x,v}}\\
      &\lesssim (1+\|\eta'\|_{\LR{1}_v})(\|u_0\|_{L^1_{x}}^m +\|S\|_{L^1_{t,x}}^m + 1) +\|\nonl'\|_{L^\infty(\supp \eta)}\|f^{\eps_3}\|_{\LR{1}_{t,x,v}}\\
      &\lesssim (1+\|\eta'\|_{\LR{1}_v})(1+\|\nonl'\|_{L^\infty(\supp \eta)})(\|u_0\|_{L^1_{x}}^m +\|S\|_{L^1_{t,x}}^m + 1).
     \end{align*}
    The compactness of $\WSR{\sigma_t}{p}(0,T;\WSR{\sigma_x}{p}(\R^d)) \hookrightarrow L^p_tL^p_x$ implies that for each $\eta$ we can choose an a.e. convergent subsequence of $u^{\eta,\eps_3}$.
    Since $\eta$ is arbitrary, and $u^{\eps_3}$ is uniformly integrable, a diagonal argument implies the existence of an a.e.\ convergent subsequence of $u^{\eta,\eps_3} \to u$.
    This allows to pass to the $\eps_3\to 0$ limit along the lines of Step 1.

    \textit{Step 4: General data.}  Let $u_0^n$, $S^n$ be bounded and smooth approximations of $u_0$, $S$, and let $u^n$ be a corresponding solution constructed in the previous step.
    Then, \eqref{eq:l1_approx} implies
      $$\esssup_{t\in[0,T]}\|u^{n_1}(t)-u^{n_2}(t)\|_{L^1_x} \lesssim \|u^{n_1}_0-u^{n_2}_0\|_{L^1_x} + \|S^{n_1}-S^{n_2}\|_{L^1_{t,x}}, $$
    and, thus,  $u^{n}$ is a Cauchy sequence in $L^\infty_tL^1_x$. This implies that there is an a.s. convergent subsequence $u^{n} \to u$. The proof can then be concluded as in Step 1.
\end{proof}


\noindent
\thanks{\textbf{Acknowledgment.}
BG acknowledges support by the Max Planck Society through the Research Group ”Stochastic Analysis in the Sciences (SAiS)”. This work was co-funded by the European Union (ERC, FluCo, grant agreement No. 101088488). Views and opinions expressed are however those of the author(s) only and do not necessarily reflect those of the European Union or of the European Research Council. Neither the European Union nor the granting authority can be held responsible for them.
}


\frenchspacing

\begin{thebibliography}{10}

\bibitem{Abe12}
Helmut Abels.
\newblock {\em Pseudodifferential and {S}ingular {I}ntegral {O}perators}.
\newblock De Gruyter Graduate Lectures. De Gruyter, Berlin, 2012.
\newblock An {I}ntroduction with {A}pplications.

\bibitem{AbG23}
Helmut Abels and Gerd Grubb.
\newblock Fractional-{O}rder {O}perators on {N}onsmooth {D}omains.
\newblock {\em J. Lond. Math. Soc. (2)}, 107(4), 2023.

\bibitem{AbK09}
Helmut Abels and Moritz Kassmann.
\newblock The {C}auchy {P}roblem and the {M}artingale {P}roblem for
  {I}ntegro-{D}ifferential {O}perators with {N}on-{S}mooth {K}ernels.
\newblock {\em Osaka J. Math.}, 46(3):661--683, 2009.

\bibitem{AAO20}
Natha\"{e}l Alibaud, Boris Andreianov, and Adama Ou\'{e}draogo.
\newblock Nonlocal {D}issipation {M}easure and {$L^1$} {K}inetic {T}heory for
  {F}ractional {C}onservation {L}aws.
\newblock {\em Comm. Partial Differential Equations}, 45(9):1213--1251, 2020.

\bibitem{AFP00}
Luigi Ambrosio, Nicola Fusco, and Diego Pallara.
\newblock {\em Functions of {B}ounded {V}ariation and {F}ree {D}iscontinuity
  {P}roblems}.
\newblock Oxford Mathematical Monographs. The Clarendon Press, Oxford
  University Press, New York, 2000.

\bibitem{BCD11}
H.\@ Bahouri, J.-Y.\@ Chemin, and R.\@ Danchin.
\newblock {\em Fourier {A}nalysis and {N}onlinear {P}artial {D}ifferential
  {E}quations}, volume 343 of {\em Grundlehren der Mathematischen
  Wissenschaften}.
\newblock Springer, Heidelberg, 2011.

\bibitem{BKM2010}
Piotr Biler, Grzegorz Karch, and R\'{e}gis Monneau.
\newblock Nonlinear {D}iffusion of {D}islocation {D}ensity and {S}elf-{S}imilar
  {S}olutions.
\newblock {\em Comm. Math. Phys.}, 294(1):145--168, 2010.

\bibitem{BDKS18}
V.\@ B{\"o}gelein, F.\@ Duzaar, R.\@ Korte, and C.\@ Scheven.
\newblock The {H}igher {I}ntegrability of {W}eak {S}olutions of {P}orous
  {M}edium {S}ystems.
\newblock {\em Advances in Nonlinear Analysis}, 2018.

\bibitem{BV2015}
Matteo Bonforte and Juan~Luis V\'{a}zquez.
\newblock A {P}riori {E}stimates for {F}ractional {N}onlinear {D}egenerate
  {D}iffusion {E}quations on {B}ounded {D}omains.
\newblock {\em Arch. Ration. Mech. Anal.}, 218(1):317--362, 2015.

\bibitem{BLS2018}
Lorenzo Brasco, Erik Lindgren, and Armin Schikorra.
\newblock Higher {H}\"{o}lder {R}egularity for the {F}ractional
  {$p$}-{L}aplacian in the {S}uperquadratic {C}ase.
\newblock {\em Adv. Math.}, 338:782--846, 2018.

\bibitem{CV2011}
Luis Caffarelli and Juan~Luis Vazquez.
\newblock Nonlinear {P}orous {M}edium {F}low with {F}ractional {P}otential
  {P}ressure.
\newblock {\em Arch. Ration. Mech. Anal.}, 202(2):537--565, 2011.

\bibitem{CdPG2023}
Pedro Cardoso, Renato de~Paula, and Patr\'{\i}cia Gon\c{c}alves.
\newblock Derivation of the {F}ractional {P}orous {M}edium {E}quation from a
  {M}icroscopic {D}ynamics.
\newblock {\em Nonlinearity}, 36(3):1840--1872, 2023.

\bibitem{CDPFMV2017}
Jos\'{e}~Antonio Carrillo, Manuel del Pino, Alessio Figalli, Giuseppe Mingione,
  and Juan~Luis V\'{a}zquez.
\newblock {\em Nonlocal and {N}onlinear {D}iffusions and {I}nteractions: {N}ew
  {M}ethods and {D}irections}, volume 2186 of {\em Lecture Notes in
  Mathematics}.
\newblock Springer, Cham; Fondazione C.I.M.E., Florence, 2017.
\newblock Lectures from the CIME Course held in Cetraro, July 4--8, 2016,
  Edited by Matteo Bonforte and Gabriele Grillo, Fondazione CIME/CIME
  Foundation Subseries.

\bibitem{ChP03}
G.-Q.\@ Chen and B.~Perthame.
\newblock Well-{P}osedness for {N}on-{I}sotropic {D}egenerate
  {P}arabolic-{H}yperbolic {E}quations.
\newblock {\em Ann. Inst. H. Poincar\'{e} Anal. Non Lin\'{e}aire},
  20(4):645--668, 2003.

\bibitem{CHJZ2022}
Li~Chen, Alexandra Holzinger, Ansgar J\"{u}ngel, and Nicola Zamponi.
\newblock Analysis and {M}ean-{F}ield {D}erivation of a {P}orous-{M}edium
  {E}quation with {F}ractional {D}iffusion.
\newblock {\em Comm. Partial Differential Equations}, 47(11):2217--2269, 2022.

\bibitem{CJ2011}
Simone Cifani and Espen~R. Jakobsen.
\newblock Entropy {S}olution {T}heory for {F}ractional {D}egenerate
  {C}onvection-{D}iffusion {E}quations.
\newblock {\em Ann. Inst. H. Poincar\'{e} C Anal. Non Lin\'{e}aire},
  28(3):413--441, 2011.

\bibitem{ClL93}
P.~Cl{\'e}ment and S.~Li.
\newblock Abstract {P}arabolic {Q}uasilinear {E}quations and {A}pplication to a
  {G}roundwater {F}low {P}roblem.
\newblock {\em Adv. Math. Sci. Appl.}, 3(Special Issue):17--32, 1993/94.

\bibitem{C17}
Matteo Cozzi.
\newblock Interior {R}egularity of {S}olutions of {N}on-{L}ocal {E}quations in
  {S}obolev and {N}ikol'skii {S}paces.
\newblock {\em Ann. Mat. Pura Appl. (4)}, 196(2):555--578, 2017.

\bibitem{dPQRV2011}
Arturo de~Pablo, Fernando Quir\'{o}s, Ana Rodr\'{\i}guez, and Juan~Luis
  V\'{a}zquez.
\newblock A {F}ractional {P}orous {M}edium {E}quation.
\newblock {\em Adv. Math.}, 226(2):1378--1409, 2011.

\bibitem{DT2014}
F\'{e}lix del Teso.
\newblock Finite {D}ifference {M}ethod for a {F}ractional {P}orous {M}edium
  {E}quation.
\newblock {\em Calcolo}, 51(4):615--638, 2014.

\bibitem{Ebm05}
C.~Ebmeyer.
\newblock Regularity in {S}obolev {S}paces for the {F}ast {D}iffusion and the
  {P}orous {M}edium {E}quation.
\newblock {\em J. Math. Anal. Appl.}, 307(1):134--152, 2005.

\bibitem{EJ2014}
J.~Endal and E.~R. Jakobsen.
\newblock {$L^1$} {C}ontraction for {B}ounded ({N}onintegrable) {S}olutions of
  {D}egenerate {P}arabolic {E}quations.
\newblock {\em SIAM J. Math. Anal.}, 46(6):3957--3982, 2014.

\bibitem{EM2023}
M.~Erceg, M.~Mi\v{s}ur, and D.~Mitrovi\'{c}.
\newblock Velocity {A}veraging for {D}iffusive {T}ransport {E}quations with
  {D}iscontinuous {F}lux.
\newblock {\em J. Lond. Math. Soc. (2)}, 107(2):658--703, 2023.

\bibitem{EM2023-2}
Marko Erceg and Darko Mitrović.
\newblock Degenerate {P}arabolic {E}quations -- {C}ompactness and {R}egularity
  of {S}olutions, 2023.

\bibitem{Ges21}
Benjamin Gess.
\newblock Optimal {R}egularity for the {P}orous {M}edium {E}quation.
\newblock {\em J. Eur. Math. Soc. (JEMS)}, 23(2):425--465, 2021.

\bibitem{GL2019}
Benjamin Gess and Xavier Lamy.
\newblock Regularity of {S}olutions to {S}calar {C}onservation {L}aws with a
  {F}orce.
\newblock {\em Ann. Inst. H. Poincar\'{e} C Anal. Non Lin\'{e}aire},
  36(2):505--521, 2019.

\bibitem{GST20}
Benjamin Gess, Jonas Sauer, and Eitan Tadmor.
\newblock Optimal {R}egularity in {T}ime and {S}pace for the {P}orous {M}edium
  {E}quation.
\newblock {\em Anal. PDE}, 13(8):2441--2480, 2020.

\bibitem{GiS19}
Ugo Gianazza and Sebastian Schwarzacher.
\newblock Self-{I}mproving {P}roperty of {D}egenerate {P}arabolic {E}quations
  of {P}orous {M}edium-{T}ype.
\newblock {\em Amer. J. Math.}, 141(2):399--446, 2019.

\bibitem{G2018}
Gerd Grubb.
\newblock Regularity in {$L_p$} {S}obolev {S}paces of {S}olutions to
  {F}ractional {H}eat {E}quations.
\newblock {\em J. Funct. Anal.}, 274(9):2634--2660, 2018.

\bibitem{Har06}
Michael Hardy.
\newblock Combinatorics of {P}artial {D}erivatives.
\newblock {\em Electron. J. Combin.}, 13(1):Research Paper 1, 13, 2006.

\bibitem{KU2011}
Kenneth~H. Karlsen and S\"{u}leyman Ulusoy.
\newblock Stability of {E}ntropy {S}olutions for {L}\'{e}vy {M}ixed
  {H}yperbolic-{P}arabolic {E}quations.
\newblock {\em Electron. J. Differential Equations}, pages No. 116, 23, 2011.

\bibitem{L1996}
Gary~M. Lieberman.
\newblock {\em Second {O}rder {P}arabolic {D}ifferential {E}quations}.
\newblock World Scientific Publishing Co., Inc., River Edge, NJ, 1996.

\bibitem{MSZ03}
Jan Mal\'{y}, David Swanson, and William~P. Ziemer.
\newblock The {C}o-{A}rea {F}ormula for {S}obolev {M}appings.
\newblock {\em Trans. Amer. Math. Soc.}, 355(2):477--492, 2003.

\bibitem{MCS01}
Celso Mart{\'{\i}}nez~Carracedo and Miguel Sanz~Alix.
\newblock {\em The {T}heory of {F}ractional {P}owers of {O}perators}, volume
  187 of {\em North-Holland Mathematics Studies}.
\newblock North-Holland Publishing Co., Amsterdam, 2001.

\bibitem{MSY2021}
Tadele Mengesha, Armin Schikorra, and Sasikarn Yeepo.
\newblock Calderon-{Z}ygmund {T}ype {E}stimates for {N}onlocal {PDE} with
  {H}\"{o}lder {C}ontinuous {K}ernel.
\newblock {\em Adv. Math.}, 383:Paper No. 107692, 64, 2021.

\bibitem{Min10}
Giuseppe Mingione.
\newblock Nonlinear {A}spects of {C}alder\'{o}n-{Z}ygmund {T}heory.
\newblock {\em Jahresber. Dtsch. Math.-Ver.}, 112(3):159--191, 2010.

\bibitem{NGS17a}
Van~Kien Nguyen and Winfried Sickel.
\newblock Isotropic and {D}ominating {M}ixed {B}esov {S}paces: a {C}omparison.
\newblock In {\em Functional {A}nalysis, {H}armonic {A}nalysis, and {I}mage
  {P}rocessing: a {C}ollection of {P}apers in {H}onor of {B}j\"{o}rn
  {J}awerth}, volume 693 of {\em Contemp. Math.}, pages 363--389. Amer. Math.
  Soc., Providence, RI, 2017.

\bibitem{Nik62}
S.\@~M. Nikol'ski\u\i.
\newblock Boundary {P}roperties of {D}ifferentiable {F}unctions of {S}everal
  {V}ariables.
\newblock {\em Dokl. Akad. Nauk SSSR}, 146:542--545, 1962.

\bibitem{Nik63b}
S.\@~M. Nikol'ski\u\i.
\newblock Functions with {D}ominant {M}ixed {D}erivative, {S}atisfying a
  {M}ultiple {H}\"older {C}ondition.
\newblock {\em Sibirsk. Mat. \v Z.}, 4:1342--1364, 1963.

\bibitem{Nik63a}
S.\@~M. Nikol'ski\u\i.
\newblock Stable {B}oundary-{V}alue {P}roblems of a {D}ifferentiable {F}unction
  of {S}everal {V}ariables.
\newblock {\em Mat. Sb. (N.S.)}, 61 (103):224--252, 1963.

\bibitem{ODI2021}
Adama Ou\'{e}draogo, Dofyniwassouani Alain~Houede, and Idrissa Ibrango.
\newblock Renormalized {S}olutions for {C}onvection-{D}iffusion {P}roblems
  {I}nvolving a {N}onlocal {O}perator.
\newblock {\em NoDEA Nonlinear Differential Equations Appl.}, 28(5):Paper No.
  55, 27, 2021.

\bibitem{Perthame2002}
B.~Perthame.
\newblock {\em Kinetic {F}ormulation of {C}onservation {L}aws}, volume~21 of
  {\em Oxford Lecture Series in Mathematics and its Applications}.
\newblock Oxford University Press, Oxford, 2002.

\bibitem{Pru02}
J.~Pr{\"u}{\ss}.
\newblock Maximal {R}egularity for {E}volution {E}quations in {$L_p$}-{S}paces.
\newblock {\em Conf. Semin. Mat. Univ. Bari}, (285):1--39 (2003), 2002.

\bibitem{RoS22a}
Nikolaos Roidos and Yuanzhen Shao.
\newblock The {F}ractional {P}orous {M}edium {E}quation on {M}anifolds with
  {C}onical {S}ingularities {I}.
\newblock {\em J. Evol. Equ.}, 22(1):Paper No. 8, 39, 2022.

\bibitem{RoS22b}
Nikolaos Roidos and Yuanzhen Shao.
\newblock Maximal {$L_q$}-regularity of {N}onlocal {P}arabolic {E}quations in
  {H}igher {O}rder {B}essel {P}otential {S}paces.
\newblock {\em Pure Appl. Funct. Anal.}, 7(3):1037--1063, 2022.

\bibitem{ScT87}
H.-J.\@ Schmeisser and H.~Triebel.
\newblock {\em Topics in {F}ourier {A}nalysis and {F}unction {S}paces}.
\newblock A Wiley-Interscience Publication. John Wiley \& Sons, Ltd.,
  Chichester, 1987.

\bibitem{Si1986}
Jacques Simon.
\newblock Compact {S}ets in the {S}pace {$L^p(0,T;B)$}.
\newblock {\em Ann. Mat. Pura Appl. (4)}, 146:65--96, 1987.

\bibitem{TaT07}
E.~Tadmor and T.~Tao.
\newblock Velocity {A}veraging, {K}inetic {F}ormulations, and {R}egularizing
  {E}ffects in {Q}uasi-{L}inear {PDE}s.
\newblock {\em Comm. Pure Appl. Math.}, 60(10):1488--1521, 2007.

\bibitem{Tri19}
Hans Triebel.
\newblock {\em Function {S}paces with {D}ominating {M}ixed {S}moothness}.
\newblock EMS Series of Lectures in Mathematics. European Mathematical Society
  (EMS), Z\"{u}rich, 2019.

\bibitem{Vaz14}
Juan~Luis V\'{a}zquez.
\newblock Barenblatt {S}olutions and {A}symptotic {B}ehaviour for a {N}onlinear
  {F}ractional {H}eat {E}quation of {P}orous {M}edium {T}ype.
\newblock {\em J. Eur. Math. Soc. (JEMS)}, 16(4):769--803, 2014.

\bibitem{PQRV17}
Juan~Luis V\'{a}zquez, Arturo de~Pablo, Fernando Quir\'{o}s, and Ana
  Rodr\'{\i}guez.
\newblock Classical {S}olutions and {H}igher {R}egularity for {N}onlinear
  {F}ractional {D}iffusion {E}quations.
\newblock {\em J. Eur. Math. Soc. (JEMS)}, 19(7):1949--1975, 2017.

\bibitem{Vyb06}
Jan Vybiral.
\newblock Function {S}paces with {D}ominating {M}ixed {S}moothness.
\newblock {\em Dissertationes Math.}, 436:73, 2006.

\bibitem{JJG2018}
Jinlong Wei, Jinqiao Duan, and Guangying Lv.
\newblock Kinetic {S}olutions for {N}onlocal {S}calar {C}onservation {L}aws.
\newblock {\em SIAM J. Math. Anal.}, 50(2):1521--1543, 2018.

\end{thebibliography}

\end{document}